\documentclass[12pt,reqno]{amsart}
\usepackage{amsmath} 

\usepackage{amsthm}
\usepackage{amssymb}
\usepackage{graphics}
\usepackage{latexsym}
\usepackage{cite}

\numberwithin{equation}{section}
\newtheorem{thm}{Theorem}[section]
\newtheorem{prop}[thm]{Proposition}
\newtheorem{lem}[thm]{Lemma}
\newtheorem{cor}[thm]{Corollary}

\theoremstyle{remark}
\newtheorem{rem}{Remark}[section]
\newtheorem{defn}{Definition}
 
\newcommand{\laplacian}{\Delta}
\newcommand{\BBB}{\mathbb}
\newcommand{\R}{{\BBB R}}
\newcommand{\Z}{{\BBB Z}}
\newcommand{\T}{{\BBB T}}

\newcommand{\C}{{\BBB C}}
\newcommand{\LR}[1]{{\langle {#1} \rangle }}
\newcommand{\lec}{{\ \lesssim \ }}
\newcommand{\gec}{{\ \gtrsim \ }}

\newcommand{\cross}{\times}

\newcommand{\vp}{\varphi}

\newcommand{\e}{\varepsilon}
\newcommand{\ta}{\tau}
\newcommand{\x}{\xi}

\newcommand{\p}{\partial}
\newcommand{\la}{\lambda}

\newcommand{\de}{\delta}
\newcommand{\om}{\omega}

\newcommand{\na}{\nabla}

\newcommand{\supp}{\operatorname{supp}}

\newcommand{\I}{\infty}

\newcommand{\EQ}[1]{\begin{equation} \begin{split} #1
 \end{split} \end{equation}}
\newcommand{\EQS}[1]{\begin{align} #1 \end{align}}
\newcommand{\EQQS}[1]{\begin{align*} #1 \end{align*}}
\newcommand{\EQQ}[1]{\begin{equation*} \begin{split} #1
 \end{split} \end{equation*}}

\newcommand{\ol}{\overline}
\newcommand{\ds}{\displaystyle}
\newcommand{\sub}{\subset}

\newcommand{\F}{\mathcal{F}}
\newcommand{\1}{{\mathbf 1}}

\newcommand{\ti}{\widetilde}
\newcommand{\ha}{\widehat}
\newcommand{\til}{\tilde}

\title[Scattering for the Zakharov system]{%Scattering and well-posedness\\
}

\author[I.  Kato]{%Isao Kato
}
\author[K. Tsugawa]{%Kotaro Tsugawa
}
\email[K. Tsugawa]{tsugawa@math.nagoya-u.ac.jp}
\email[Isao Kato]{kato.isao@f.mbox.nagoya-u.ac.jp}

\begin{document}

\begin{center}
{\bf SCATTERING AND WELL-POSEDNESS\\
FOR THE ZAKHAROV SYSTEM AT A CRITICAL SPACE\\
IN FOUR AND MORE SPATIAL DIMENSIONS}

\bigskip \bigskip {\sc Isao Kato}

\smallskip {\small Graduate School of Mathematics, Nagoya University, 
Chikusa-ku, Nagoya, 464-8602, Japan}

\bigskip {\sc Kotaro Tsugawa}

\smallskip {\small Graduate School of Mathematics, Nagoya University
Chikusa-ku, Nagoya, 464-8602, Japan}
\end{center}

\begin{abstract}
We study the Cauchy problem for the Zakharov system in spatial dimension $d\ge 4$ with initial datum 
$\bigl(u(0), n(0), \p_t n(0)\bigr)\in H^k(\R^d)\times \dot{H}^l(\R^d)\times \dot{H}^{l-1}(\R^d)$.  
According to Ginibre, Tsutsumi and Velo (\!\!~\cite{GTV}), 
the critical exponent of $(k,l)$ is $\bigl((d-3)/2,(d-4)/2\bigr).\ $
We prove the small data global well-posedness and the scattering at the critical space.
It seems difficult to get the crucial bilinear estimate only by applying the $U^2,\ V^2$ type spaces introduced by Koch and Tataru (\!\!~\cite{KT1}, \cite{KT2}).
To avoid the difficulty, we use an intersection space of $V^2$ type space and the space-time Lebesgue space $E:=L^2_tL_x^{2d/(d-2)}$, which is related to the endpoint Strichartz estimate.
\end{abstract}
\maketitle
%\setcounter{page}{000}

%%%%%%%%%%%%%%%%%%%%%%%%%%%%%%%%%%%%%%%%%%%%%%%%%%%%%%%%%%%%%%%%%%%%%%%%%%%%%%%%%
%%%%%%%%%%%%%%%%%%%%%%%%%%%%%%%%%%%%%%%%%%%%%%%%%%%%%%%%%%%%%%%%%%%%%%%%%%%%%%%%%
%%%%%%%%%%%%%%%%%%%%%%%%%%%%%%%%%  Section 1  %%%%%%%%%%%%%%%%%%%%%%%%%%%%%%%%%%%%%
%%%%%%%%%%%%%%%%%%%%%%%%%%%%%%%%%%%%%%%%%%%%%%%%%%%%%%%%%%%%%%%%%%%%%%%%%%%%%%%%%
%%%%%%%%%%%%%%%%%%%%%%%%%%%%%%%%%%%%%%%%%%%%%%%%%%%%%%%%%%%%%%%%%%%%%%%%%%%%%%%%%

\section{Introduction}
We consider the Cauchy problem for the Zakharov system:
\EQS{
 \begin{cases}
  i\p_t u + \laplacian u = nu, \ \ \ \ \ \ \ \ \ \ \ \ \ \ \ \ \  t \in \R,\ x \in \R^d,  \\
  \p_t^2 n - \laplacian n = \laplacian |u|^2, \ \ \ \ \ \ \ \ \ \ \ \ \ \ t \in \R,\ x \in \R^d, \\
  \bigl(u(0),n(0),\p_t n(0)\bigr) = (u_0,n_0,n_1)\in H^k(\R^d)\times \dot{H}^l(\R^d)\times \dot{H}^{l-1}(\R^d),      \label{d-Z}
 \end{cases}
}
where $u=u(t, x)$ is complex valued, the slowly varying envelope of electric field and 
$n=n(t, x)$ is real valued, the deviation of ion density from its mean background density.
\eqref{d-Z} describes the Langmuir turbulence in a plasma.
We consider well-posedness for \eqref{d-Z} in spatial dimension $d\ge 4.$  
\eqref{d-Z} does not have scaling invariant transformation because of the difference of dilation transformations for the linear wave equation and the  Schr\"odinger equation.
However,  in \cite{GTV}, Ginibre, Tsutsumi and Velo
introduced a critical exponent for \eqref{d-Z} which corresponds to 
the scaling criticality in the following sense.
We transform $n$ into $n_{\pm}$ as $n_{\pm}:=n\pm i\om^{-1}\p_t n,\ \om := \sqrt{-\laplacian }$.  
Then \eqref{d-Z} is rewritten into
\EQS{\label{redn}
 \begin{cases}
  i\p_t u + \laplacian u = u (n_++ n_-)/2, \ \ \ \ \ \ \ \ \ \ \ \ \ \ \ \ \  t \in \R,\ x \in \R^d,  \\
   (i\p_t \mp \om) n_{\pm} = \pm \om |u|^2, \ \ \ \ \ \ \ \ \ \ \ \ \ \ \ \ \ \ \ \ \ \ \ \ t \in \R,\ x \in \R^d, \\
  \bigl(u(0),n_{+}(0),n_{-}(0)\bigr) = (u_0,n_{+0},n_{-0}).
 \end{cases}
}
In the second equation of \eqref{redn}, if we disregard the second term of the left-hand side,
then \eqref{redn} is invariant under the dilation
\EQQ{
u\to u_\la=\la^{3/2}u(\la x, \la^2 t),\ \ \ \ n\to n_{\pm \la}=\la^2n_{\pm}(\la x,\la^2 t), 
}
and the the scaling critical exponent is $(k, l)=\bigl((d-3)/2, (d-4)/2\bigr)$.
Our main result is the small data global well-posedness and the scattering for \eqref{redn} at the critical exponent in spatial dimension $d\ge 4$. 
%%%%%%%%%%%%%%%%%%%%%%%%%%%%%%
%%%%%  SDGW Main Theorem %%%%%
%%%%%%%%%%%%%%%%%%%%%%%%%%%%%%
\begin{thm} \label{SDGW}
Let $d \ge 4, k=(d-3)/2, l=(d-4)/2$. 
Then \eqref{redn} is globally well-posed for small data in $H^k(\R^d) \cross \dot{H}^l(\R^d) \cross \dot{H}^{l}(\R^d)$  (resp. $H^k(\R^d) \cross {H}^l(\R^d) \cross {H}^{l}(\R^d)$). Moreover, the solution scatters
 in this space.
%%%%%%%%%%%%%%%%%%%%%%%%%%%%%%%%%%%%%%%%%%%%%%%%%%%%%%%%%%%%%%%%%%%%%%%%%%%%%%%%%%%%%%%%%%%%%%   
%%%%%%%%%%%%%%%%%%%%%%%%%%%%%%%%%%%%%%%%%%%%%%%%%%%%%%%%%%%%%%%%%%%%%%%%%%%%%%%%%%%%%%%%%%%%%%
\end{thm}
\begin{rem}
Note that $(n_+,n_-)\in \dot{H}^{l}(\R^d) \cross \dot{H}^{l}(\R^d)$ (resp.  ${H}^l(\R^d) \cross {H}^{l}(\R^d)$)
is equivalent to $(n,\p_t n)\in \dot{H}^l(\R^d) \cross \dot{H}^{l-1}(\R^d)$ (resp. ${H}^l(\R^d) \cross \om{H}^{l}(\R^d)$).
If we use the transform $n_{\pm}:=n\pm i\om_1^{-1} \p_t n$ with $\om_1:=\sqrt{1-\laplacian}$ instead of $n_{\pm}:=n\pm i\om^{-1} \p_t n$,
then $(n_+,n_-)\in {H}^l(\R^d) \cross {H}^{l}(\R^d)$ is equivalent to
$(n,\p_t n) \in {H}^l(\R^d) \cross {H}^{l-1}(\R^d)$ and the second  equation of \eqref{redn} is rewritten  into
\EQ{\label{trans1}
(i\p_t \mp \om_1) n_{\pm} = \mp \om_1^{-1} \laplacian |u|^2 \mp \om_1^{-1}(n_++n_-)/2.
}
This transform was used in \cite{GTV} to study the local well-posedness.
We can deal with the first term of the right-hand side of \eqref{trans1} in the same manner as \eqref{redn}. The second term $\om_1^{-1}(n_++n_-)/2$ is harmless when we consider the local well-posedness.
However, we do not know how to deal with it when we consider the global well-posedness. For this reason, the global well-posedness of \eqref{d-Z} in $H^k(\R^d) \cross H^l(\R^d) \cross H^{l-1}(\R^d)$ is still open problem.
\end{rem}
For more precise statement of Theorem \ref{SDGW}, see Propositions \ref{main_prop1}, \ref{main_prop2}.
Here, we briefly mention the known results for the Cauchy problem for \eqref{d-Z}.
There are many results for $3\ge d\ge 1$.
The local and global well-posedness and related results were obtained in 1990s.  For the case on $\R^d$, see \cite{BH, BHHT,BC, CHT, FPZ, GTV, GV4, GLNW, GN, GNW, HPS, Ho, KPV2, MN2, OT, OT4, P, Shi}.
For the case on $\T^d$, see \cite{Bo2, Ki, Ki2, Tak}.
The scattering results were obtained only after 2010 in \cite{GLNW, GN, GNW, HPS}.
All these results are for the sub critical case.
For $d \ge 4$, Ginibre, Tsutsumi and Velo \cite{GTV} proved the local well-posedness of \eqref{d-Z}
when the initial data is in $H^k(\R^d) \times H^l(\R^d) \times H^{l-1}(\R^d)$
with $2k>l+(d-2)/2, l>(d-4)/2, l+1 \ge k \ge l$, which is the sub critical case.
Recently,  Bejenaru, Guo, Herr and Nakanishi \cite{BGHN} have proved the small data global well-posedness and the scattering in a range of $(k,l)$ for $d=4$, which includes
the critical case $(k,l)=(1/2,0)$ and the energy space $(k,l)=(1,0)$.

The main difficulty in the study of the well-posedness of the Zakharov system arises from so called ``derivative loss''.
The both nonlinear terms of \eqref{redn} have
a half derivative loss when $k=l+1/2$.
To recover the derivative loss, Ginibre, Tsutsumi and Velo \cite{GTV} applied the Fourier restriction norm method, which was introduced by Bourgain \cite{Bo}.
Bejenaru, Guo, Herr and Nakanishi \cite{BGHN} used the normal form reduction and transformed \eqref{redn} into a system which does not have derivative loss.
Our proof is more direct than their proof.
We use the $U^2,\ V^2$ type spaces, which were introduced by
Koch and Tataru \cite{KT1}, \cite{KT2} and applied by
Hadac, Herr and Koch \cite{HHK} to the small data global well-posedness
and the scattering for the KP-II equation at the scale critical space.
There are two merits for using these function spaces.
One is that we can recover the derivative loss,
by combining Lemma \ref{Recovery} and \eqref{mod} in Proposition \ref{modulation}.
The other is that we can employ the Strichartz estimate (see Proposition \ref{Strichartz}) by Corollary \ref{Strichartz-S} and we gain some integrability.
Though the Fourier restriction norm $X^{s,1/2+\e}$ also have the same merits,
it seems difficult to apply it for the critical case.
Because the estimate has small loss of integrability if we take $\e \le 0$ when we employ the Strichartz estimate (see Lemma 2.4 in \cite{GTV}) and we can recover only $1/2-\e$ derivative loss if we take $\e>0$ and it is not enough for our purpose.
This is the reason why the results in \cite{GTV} is only for sub critical case and we use not the Fourier restriction norm but the $U^2,\ V^2$ type spaces.

There is another difficulty for the Zakharov system.
It is caused by the difference of the dilation scale of the Schr\"{o}dinger equation and the wave equation.
It is known that the effect by oscillatory integral for the Schr\"{o}dinger equation works more effectively than that of the wave equation.
For instance, for $d=4$, by the H\"older inequality and the Bernstein inequality, we have 
\EQQ{
\| ( P_N e^{it\laplacian}f) (P_N e^{\mp it\sqrt{\laplacian}} g) \|_{L^2_t L^2_x}
  &\lec\| P_N e^{it\laplacian} f\|_{L^2_tL^\infty_x}
    \| P_N e^{\mp it\sqrt{\laplacian}} g \|_{L^\infty_t L^2_x}\\
&\lec N\| P_N e^{it\laplacian} f\|_{L^2_tL^4_x}
    \| P_N e^{\mp it\sqrt{\laplacian}} g \|_{L^\infty_t L^2_x}\\
&\lec N\|P_N f\|_{L^2_x} \|P_N g\|_{L^2_x},
}
if we use the endpoint Strichartz estimate for the Schr\"odinger equation, and
\EQQ{
\| ( P_N e^{it\laplacian}f) (P_N e^{\mp it\sqrt{\laplacian}} g) \|_{L^2_t L^2_x}
  &\lec\| P_N e^{it\laplacian} f\|_{L^\infty_tL^2_x}
    \| P_N e^{\mp it\sqrt{\laplacian}} g \|_{L^2_t L^\infty_x}\\
&\lec \| P_N e^{it\laplacian} f\|_{L^\infty_tL^2_x}
    N^{2/3}\| P_N e^{\mp it\sqrt{\laplacian}} g \|_{L^2_t L^6_x}\\
&\lec N^{3/2}\|P_N f\|_{L^2_x} \|P_N g\|_{L^2_x},
}
if we use the endpoint Strichartz estimate for the wave equation.
The former estimate is $1/2$ derivative better than the latter.
Therefore, to estimate the quadratic nonlinear term, we use the endpoint Strichartz estimate for the Schr\"odinger equation,
that is to say the case of $(p_1,q_1)=(2,2d/(d-2))$ in Proposition \ref{Strichartz}.
This causes the following problem: if we use the $U^2$ type function space and follow the argument by Hadac, Herr and Koch \cite{HHK},
then by duality argument (see Proposition \ref{U^2norm}) we need to estimate $L^2_t L^{2d/(d-2)}_x$ norm by the $V^2$ type norm.
However, we can not get such estimate by Corollary \ref{Strichartz-S} because the $V^2$ type norm is slightly weaker than $U^2$ type norm.
For this reason, we need the function space 
weaker than the $U^2$ type and stronger than the $V^2$ type.
For that purpose, we use an intersection space of $V^2$ type space and  $E:=L^2_t L^{2d/(d-2)}_x$.
See the definition of $\|u\|_{X^k_S}$ in Definition \ref{defX},
which is the main idea in the present paper.
Note that the $L^4$ Strichartz estimate was used and this difficulty was not caused for the KP-II equation in \cite{HHK}.
 
Finally, we refer to the plan of the rest of the paper.
We introduce function spaces, their properties and some lemmas in Section 2.
In Section 3, we derive the key bilinear estimate for the homogeneous case, Proposition \ref{B.E.homo}.
As a corollary, we also prove the bilinear estimate for the inhomogeneous case, Corollary \ref{B.E.S.W.}.
In Section 4, we mention the detail of main theorem and its proof.

\section*{Acknowledgement}
The second author is supported by JSPS KAKENHI Grant Number 25400158.

%%%%%%%%%%%%%%%%%%%%%%%%%%%%%%%%%%%%%%%%%%%%%%%%%%%%%%%%%%%%%%%%%%%%%%%%%%%%%%%%%%%%%%%
%%%%%%%%%%%%%%%%%%%%%%%%%%%%%%%%%%%%%%%%%%%%%%%%%%%%%%%%%%%%%%%%%%%%%%%%%%%%%%%%%%%%%%%
%%%%%%%%%%%%%%%%%%%%%%%%%%%%%%%%%%%%%  Section 2  %%%%%%%%%%%%%%%%%%%%%%%%%%%%%%%%%%%%%%%
%%%%%%%%%%%%%%%%%%%%%%%%%%%%%%%%%%%%%%%%%%%%%%%%%%%%%%%%%%%%%%%%%%%%%%%%%%%%%%%%%%%%%%%
%%%%%%%%%%%%%%%%%%%%%%%%%%%%%%%%%%%%%%%%%%%%%%%%%%%%%%%%%%%%%%%%%%%%%%%%%%%%%%%%%%%%%%%

\section{notations and preliminary lemmas}
In this section, we prepare some lemmas, propositions and notations to prove the main theorem.
Notations related to $U^p$ and $V^p$ spaces are based on the definition in \cite{HHK} and \cite{HHK2}.
$A\lec B$ means that there exists $C>0$ such that $A \le CB.$ 
Also, $A\sim B$ means $A\lec B$ and $B\lec A.$  
Let $u=u(t,x).\ \F_t u,\ \F_x u$ denote the Fourier transform of $u$ in time, space, respectively. 
$\F_{t,\, x} u = \ha{u}$ denotes the Fourier transform of $u$ in space and time.    
%%%%%%%%%%%%%%%%%%%%%%%%%%%%%%%%
%%%%% Definition of U^p space %%%%%
%%%%%%%%%%%%%%%%%%%%%%%%%%%%%%%%
Let $\mathcal{Z}$ be the set of finite partitions $-\I=t_0<t_1<\cdots <t_K = \I$ and let $\mathcal{Z}_0$ 
be the set of finite partitions $-\I<t_0<t_1<\cdots <t_K \le \I$. 
\begin{defn}
Let $1\le p< \I.$ For $\{t_k\}_{k=0}^K \in \mathcal{Z}$ and $\{ \phi_k\}_{k=0}^{K-1}\subset L^2_x$ with   
$\sum_{k=0}^{K-1} \|\phi_k \|_{L^2_x}^p=1$ and $\phi_0=0$, we call the function $a : \R \to L^2_x$ given by 
\EQQS{
 a=\sum_{k=1}^K \1_{[t_{k-1},\, t_k)}\phi_{k-1} 
}
a $U^p$-atom. Furthermore, we define the atomic space 
\EQQS{
 U^p:=\biggl{\{} u=\sum_{j=1}^{\I}\la_j a_j \, \Bigl| \, a_j : U^p \text{-atom} , \la_j \in \C \ such\ that\ \sum_{j=1}^{\I}|
          \la_j|< \I \biggr{\}}
}
with norm 
\EQQS{
 \| u\|_{U^p}:=\inf \biggl{\{} \sum_{j=1}^{\I}|\la_j| \, \Bigl| \, u=\sum_{j=1}^{\I}\la_j a_j, \la_j\in \C, a_j : U^p \text{-atom}
                   \biggr{\}}.
} 
\end{defn} 
%%%%%%%%%%%%%%%%%%%%%%%%%%%%%%%%
%%%%% Properties of U^p space %%%%%
%%%%%%%%%%%%%%%%%%%%%%%%%%%%%%%%
\begin{prop}
Let $1\le p<q<\I.$ \\
(i) $U^p$ is a Banach space. \\
(ii) The embeddings $U^p\subset U^q\subset L^{\I}_t(\R;L^2_x)$ are continuous. \\
(iii) For $u\in U^p$, it holds that $\lim_{t\to t_{0}+}\|u(t)-u(t_{0})\|_{L^2_x}=0,$ i.e. every $u\in U^p$ is right-continuous. \\
(iv) The closed subspace $U^p_c$ of all continuous functions in $U^p$ is a Banach space.
\end{prop}
The above proposition is in ~\cite{HHK} (Proposition 2.2).  
%%%%%%%%%%%%%%%%%%%%%%%%%%%%%%%
%%%%% Definition of V^p space %%%%%
%%%%%%%%%%%%%%%%%%%%%%%%%%%%%%%
\begin{defn}\label{def_of_V}
Let $1\le p<\I.$ We define $V^p$ as the normed space of all functions $v:\R\to L^2_x$ such that 
$\lim_{t\to \pm \I}v(t)$ exist and for which the norm 
\EQQS{
\| v\|_{V^p}:=\sup_{\{ t_k\}_{k=0}^K\in \mathcal{Z}}\Bigl(\sum_{k=1}^K\| v(t_k)-v(t_{k-1})\|_{L^2_x}^p\Bigr)^{1/p}  
} 
is finite, where we use the convention that $v(-\I):=\lim_{t\to -\I}v(t)$ and $v(\I):=0.$
Note that $v(\infty)$ does not necessarily coincide with the limit at $\infty$.
Likewise, let $V_-^p$ denote the closed subspace of all $v \in V^p$ with $\lim_{t\to -\I}v(t)=0.$ 
\end{defn}
For the definitions of $V^p$ and $V^p_-$, see the erratum ~\cite{HHK2}.

%%%%%%%%%%%%%%%%%%%%%%%%%%%%%%%%
%%%%% Properties of V^p space %%%%%
%%%%%%%%%%%%%%%%%%%%%%%%%%%%%%%%
\begin{prop} \label{embedding}
Let $1\le p<q<\I.$\\
(i) Let $v:\R \to L^2_x$ be such that 
\EQQ{
% \| v\|_{V^p_0}:= 
\sup_{ \{ t_k\}_{k=0}^K\in \mathcal{Z}_0, t_K<\infty } \Bigl( \sum_{k=1}^K\| v(t_{k})-v(t_{k-1})\|_{L^2_x}^p\Bigr)^{1/p}
}
is finite. Then, it follows that $v(t_0^+):=\lim_{t\to t_0+} v(t)$ exists for all $t_0\in [-\I,\I)$,
$v(t_0^{-}):=\lim_{t\to t_0-} v(t)$ exists for all $t_0\in (-\I,\I]$.\\
%and moreover, 
%\EQQS{
% \| v\|_{V^p}=\| v\|_{V^p_0}.
%}
(ii) We define the closed subspace $V^p_{rc}\, (V^p_{-,\, rc})$ of all right-continuous $V^p$ functions ($V^p_-$ functions). 
The spaces $V^p,\ V^p_{rc},\ V^p_-$ and $V^p_{-,\, rc}$ are Banach spaces. \\
(iii) The embeddings $U^p\subset V^p_{-,\, rc}\subset U^q$ are continuous. \\
(iv) The embeddings $V^p\subset V^q$ and $V^p_-\subset V^q_-$ are continuous. 
\end{prop}
Note that the embedding in $(iii)$ is not consistent with the convention $v(+\infty)=0$ in Definition \ref{def_of_V} unless $v$ is discontinuous at $+\infty$.
For the proof of Proposition \ref{embedding}, see ~\cite{HHK} (Proposition 2.4 and Corollary 2.6).
Precisely, the statement of Proposition 2.4 $(i)$ in ~\cite{HHK} is for the partition $\{ t_k\}_{k=0}^K \in \mathcal{Z}_0$.
But, we can easily check that $(i)$ above is also true for the partition $\{ t_k\}_{k=0}^K\in \mathcal{Z}_0$ with $t_K<\infty$.

%%%%%%%%%%%%%%%%%%%%%%%%%%%%%%%%%%%%
%%%%%  Notation of operators Q_i^A  %%%%%
%%%%%%%%%%%%%%%%%%%%%%%%%%%%%%%%%%%%
Let $\{ \F_{x}^{-1}[\varphi_n](x)\}_{n\in \Z}\subset \mathcal{S}(\R^d)$ be the Littlewood-Paley decomposition with respect 
to $x$, that is to say
\EQQS{
\begin{cases}
 \varphi(\xi) \ge 0, \\
 \supp \varphi(\xi) = \{ \xi \,|\, 2^{-1} \le |\xi| \le 2\}, 
\end{cases}
}
\EQQS{
 \varphi_{n}(\xi) := \varphi(2^{-n}\xi),\ \sum_{n=-\I}^{\I}\varphi_{n}(\xi)=1\ (\, \xi \neq 0),\ 
 \psi(\xi) := 1-\sum_{n=0}^{\I}\varphi_{n}(\xi).
} 
Let $N=2^n\ (n \in \Z)$ be dyadic number. $P_N$ and $P_{<1}$ denote 
\EQQS{
 &\F_{x}[P_N f](\xi):=\varphi(\xi/N)\F_x[f](\xi)=\varphi_n(\xi)\F_x[f](\xi), \\
 &\F_{x}[P_{<1} f](\xi):=\psi(\xi)\F_x[f](\xi).
} 
Similarly, let $Q_N$ be 
\EQQS{
 \F_t [Q_N g](\ta):=\phi(\ta/N)\F_t[g](\ta)=\phi_n(\ta)\F_t[g](\ta),
}
where  $\{ \F_{t}^{-1}[\phi_{n}](t)\}_{n\in \Z}\subset \mathcal{S}(\R)$ be the Littlewood-Paley decomposition with respect 
to $t$ that is to say, $\phi_n$ is defined by the same manner as $\vp_n$ with $d=1$.
%%%%%%%%%%%%%%%%%%%%%%%%%%%%%%%%%%%%%%%%%%%%
Let $S(t)=\exp \{it \laplacian \}: L^2_x\to L^2_x$ be the Schr\"{o}dinger unitary operator such that 
$\F_x[S(t)u_0](\xi ) = \exp \{-it |\xi|^2\}\, \F_x[u_0](\xi ).$ 
Similarly, we define the wave unitary operator $W_{\pm}(t)=\exp \{\mp it (-\laplacian)^{1/2}\}:L^2_x\to L^2_x$ such that 
$\F_x[W_{\pm}(t)n_0](\xi ) = \exp \{\mp it |\xi |\}\, \F_x[n_0](\xi ).$
\begin{defn}
We define \\ 
 $(i)\, U^p_S=S(\cdot)U^p$ with norm $\| u\|_{U^p_S}=\|S(-\cdot )u\|_{U^p},$\\
 $(ii)\, V^p_S=S(\cdot)V^p$ with norm $\| u\|_{V^p_S}=\|S(-\cdot )u\|_{V^p}.$\\
For dyadic numbers $N,\ M$, 
\EQQS{
 Q^S_N  := S(\cdot)Q_N S(-\cdot),\ Q_{\ge M}^S:=\sum_{N\ge M}Q^S_N,\ Q_{<M}^S:=Id-Q_{\ge M}^S.
}
Here summation over $N$ means summation over $n\in \Z$. 
Similarly, we define $U^p_{W_{\pm}}$, $V^p_{W_{\pm}}$, $Q_N^{W_{\pm}}$, $Q_{<M}^{W_{\pm}}$ and $Q_{\ge M}^{W_{\pm}}.$ 
\end{defn}
\begin{rem} \label{embed}
For $L^2_x$ unitary operator $A=S$ or $W_{\pm},$ 
\EQQS{
 U^2_A \subset V^2_{-,\, rc,\, A} \subset L^{\I}(\R; L^2_x)
}
holds by Proposition 2.1 $(ii)$ and Proposition \ref{embedding} $(iii)$. 
\end{rem}
%%%%%%%%%%%%%%%%%%%%%%%%%%%%%%%%%%%%%%%%%%%%%%%%%%%%%%%%%%%%%%%%%%%%%%%%%%%%%%%%%%%%%%%%%%%%%%%%%%%%
%%%%%%%%%%%%%%%%%%%%%%%%%%%%%%%%%%%%%%%  Definition of X^k_S  %%%%%%%%%%%%%%%%%%%%%%%%%%%%%%%%%%%%%%%%%%
%%%%%%%%%%%%%%%%%%%%%%%%%%%%%%%%%%%%%%%%%%%%%%%%%%%%%%%%%%%%%%%%%%%%%%%%%%%%%%%%%%%%%%%%%%%%%%%%%%%%
\begin{defn} \label{defX}
For the Schr\"{o}dinger equation, we define 
$X^k_S$ as the closure of all $u \in C(\R; H^k_x(\R^d)) \cap \LR{\na _x}^{-k}V^2_{-,\, rc,\, S}$ such that   
\EQQS{
  \|u\|_{X^k_S}&:=\|u\|_{Y^k_S}+\|u\|_{E^k} < \I,\ 
  \|u\|_{Y^k_S}:=\|P_{<1}u\|_{V^2_S}+\Bigl(\sum_{N\ge 1} N^{2k}\|P_N u\|^2_{V^2_S}\Bigr)^{1/2} ,\\ 
  \|u\|_{E^k}&:=\|P_{<1}u\|_E+\Bigl(\sum_{N\ge 1} N^{2k}\|P_N u\|^2_E\Bigr)^{1/2}
}
with respect to the $\|\cdot\|_{X^k_S}$ norm,
where $E:=L_t^2L_x^{2d/(d-2)}$.
For the wave equation, we define 
\EQQS{
 &\|n\|_{\dot{Z}^l_{W_{\pm}}}:=\Bigl(\sum_N N^{2l}\| P_N n\|^2_{U^2_{W_{\pm}}}\Bigr)^{1/2},\,
 \|n\|_{Z^l_{W_{\pm}}}:=\|P_{<1}n\|_{U^2_{W_{\pm}}}+\Bigl(\sum_{N\ge 1} N^{2l}\|P_N n\|^2_{U^2_{W_{\pm}}}\Bigr)^{1/2}, \\
 &\|n\|_{\dot{Y}^l_{W_{\pm}}}:=\Bigl(\sum_N N^{2l}\| P_N n\|^2_{V^2_{W_{\pm}}}\Bigr)^{1/2},\, 
 \|n\|_{Y^l_{W_{\pm}}}:=\|P_{<1}n\|_{V^2_{W_{\pm}}}+\Bigl(\sum_{N\ge 1} N^{2l}\|P_N n\|^2_{V^2_{W_{\pm}}}\Bigr)^{1/2}.
}
\end{defn}
\begin{defn}
For a Hilbert space $H$ and a Banach space $X\subset C(\R; H)$, we define
\EQQS{
 &B_r(H):=\{ f \in H\, |\, \|f\|_H \le r \}, \\
 &X([0,T]):=\{ u \in C([0,T];H)\, |\, \exists \til{u} \in X , \til{u}(t)=u(t), t \in [0,T] \} 
}
endowed with the norm $\|u\|_{X([0,T])}=\inf \{ \|\til{u}\|_X |\,  \til{u}(t)=u(t), t \in [0,T]\}.$
\end{defn}

\begin{lem} \label{estX}
Let $a \ge 0$. Then for $A=S$ or $W_{\pm}$, it holds that 
\EQQS{
 \|\LR{\na _x}^a f\|_{V^2_A} \lec \|f\|_{Y^a_A}. 
}
\end{lem}
\begin{proof}
By $\ L^2_x$ orthogonality,
\EQQS{ 
  \|\LR{\na _x}^a f\|_{V^2_A}^2 
    &\lec \sup_{\{t_i\}_{i=0}^I \in \mathcal{Z}} 
              \sum_{i = 1}^I \Bigl(\|P_{<1}\bigl( A(-t_i)f(t_i)-A(-t_{i-1})f(t_{i-1})\bigr)\|_{L^2_x}^2 \notag \\
    & \ \ \ \ \ \ \ \ \ \              
              + \sum_{N \ge 1} N^{2a} \|P_N\bigl(A(-t_i)f(t_i)-A(-t_{i-1})f(t_{i-1})\bigr)\|_{L^2_x}^2 \Bigr) \notag \\
    &\lec \sup_{\{t_i\}_{i=0}^I \in \mathcal{Z}} 
              \sum_{i=1}^I \|A(-t_i)P_{<1}f(t_i)-A(-t_{i-1})P_{<1}f(t_{i-1})\|_{L^2_x}^2 \notag \\
    & \ \ \ \ \ \ \ \ \ \     
              + \sum_{N \ge 1} N^{2a} \sup_{\{t_i\}_{i=0}^I \in \mathcal{Z}} 
                   \sum_{i=1}^I \|A(-t_i)P_N f(t_i)-A(-t_{i-1})P_N f(t_{i-1})\|_{L^2_x}^2 \notag \\
    &\lec \|f\|_{Y^a_A}^2.
}
\end{proof}
\begin{rem} \label{estx}
Similarly, we see 
\EQQS{
 \bigl{\|} |\na _x|^a f\|_{V^2_A} \lec \|f\|_{\dot{Y}^a_A}.  
}
\end{rem}
For the proof of the following propositions, see Proposition 2.7, Theorem 2.8 and Proposition 2.10 in ~\cite{HHK}. 
%%%%%%%%%%%%%%%%%%%%%%%%%%%%%%%%%%%%%%%%%%%%%%%
%%%%%%%%%%  Definition of Bilinear form   %%%%%%%%%%%
%%%%%%%%%%%%%%%%%%%%%%%%%%%%%%%%%%%%%%%%%%%%%%%  
\begin{prop} \label{Bform}
Let $1<p, p'<\infty$ satisfy $1/p+1/p'=1$. 
For $u \in U^p$ and $v \in V^{p'}$ and a partition $t := \{t_i\}_{i=0}^{I} \in \mathcal{Z}$ we define 
\EQQS{
 B_{t}(u,v) := \sum_{i=1}^I \LR{u(t_{i-1}),v(t_i)-v(t_{i-1})}_{L^2_x}. 
}
There is a unique number $B(u,v)$ with the property that for all $\e >0$ there exists $t \in \mathcal{Z}$ such that for every 
$t' \supset t$ it holds 
\EQQS{
 |B_{t'}(u,v)-B(u,v)|<\e ,  
}
and the associated bilinear form 
\EQQS{
 B: U^p \cross V^{p'} \ni (u,v) \mapsto B(u,v) \in \C
}
satisfies the estimate 
\EQQS{
 |B(u,v)| \le \|u\|_{U^p} \|v\|_{V^{p'}}.
}
\end{prop}
\begin{prop}\label{prop_dual}
Let $1<p<\infty$. We have
\EQQ{
(U^p)^* =V^{p'}
}
in the sense that
\EQQ{
T : V^{p'}\to (U^p)^*, \ \ T(v):=B(\cdot,v)
}
is an isometric isomorphism.
\end{prop}
\begin{prop}\label{prop_B}
Let $1<p<\infty$, $u\in V^1_{-}$ be absolutely continuous on compact intervals
and $v\in V^{p'}$. Then,
\EQQ{
B(u,v)=-\int_{-\infty}^\infty \LR{u'(t),v(t)}_{L^2}\, dt.
}
\end{prop}
By Propositions  \ref{prop_dual}, \ref{prop_B}, we have the following proposition (see also Remark 2.11 in \cite{HHK}).
%%%%%%%%%%%%%%%%%%%%%%%%%%%%%%% 
%%%%%%%%%%%%%%%%%%%%%%%%%%%%%%%%
%%%%%  U^2, V^2 norm formula  %%%%%
%%%%%%%%%%%%%%%%%%%%%%%%%%%%%%%%
\begin{prop}  \label{U^2norm}
Let $u\in V^1_{-,rc} \sub U^2$ be absolutely continuous on compact intervals.
Then, $\|u\|_{U^2}=\ds \sup_{v \in V^2,\, \|v\|_{V^2}=1}\Bigl|\int_{-\I}^{\I}\LR{u'(t), v(t)}_{L^2_x} dt\Bigr|.$
\end{prop}
%\begin{rem}
%$\|v\|_{V^2}=\ds \sup_{u : U^2\text{-atom}}\Bigl|\int_{-\I}^{\I}\LR{u'(t), v(t)}_{L^2_x} dt\Bigr|.$        
%\end{rem}
By the proposition above, we immediately have the following corollary.
\begin{cor}  \label{U^2_A}
Let $A=S$ or $W_{\pm}$ and $u \in V^1_{-, rc, A} \subset U^2_A$ be absolutely continuous on compact intervals.
Then, 
\EQQS{
 \|u\|_{U^2_A} = \sup_{v \in V^2_A,\, \|v\|_{V^2_A}=1}
                       \Bigl|\int_{-\I}^{\I} \LR{A(t)\bigl(A(-\cdot )u\bigr)'(t), v(t)}_{L^2_x} dt\Bigr|.
}
\end{cor}
For the following remark, see Remark 2.12 in \cite{HHK}.
\begin{rem}  \label{V^2norm}
For $v\in V^2$, it holds that
\EQQ{
\|v\|_{V^2}=\sup _{u; U^2\text{-atom}} |B(u,v)|.
}
\end{rem}
\begin{prop}\label{prop_Bform2}
Let $1<p<\I, v\in V^1_-$ be absolutely continuous on compact intervals and $u$ be a $U^{p'}$-atom. Then,
\EQ{
B(u,v)=\int_{-\I}^\I \LR{u(t), v'(t)}_{L^2_x}dt-\lim_{t\to\I}\LR{u(t),v(t)}_{L^2_x}.  \label{dual_eq1}
}
\end{prop}
\begin{proof}
By Proposition \ref{embedding} $(iv)$, we have $v\in V^p$.
Therefore, the left-hand side of \eqref{dual_eq1} makes sense.
From our assumption, it follows that $v'\in L^1(\R;L^2_x)$ with $\|v'\|_{L^1(\R;L^2_x)}\le \|v\|_{V^1}<\I$ and
$$u=\sum_{k=1}^K \1_{[t_{k-1},t_k)}\phi_{k-1}$$
with $\{t_k\}_{k=0}^K\in \mathcal{Z}$, $\{\phi_k\}_{k=0}^{K-1}\subset L^2_x$, $\sum_{k=0}^{K-1}\|\phi_k\|_{L^2_x}^{p'}= 1$ and $\phi_0=0$.
By the definition of $B$, for any $\e>0$, there exists $\ti{t}=\{\ti{t}_k\}_{k=0}^N\in \mathcal{Z}$
such that for any $\mathcal{Z} \ni t'=\{t'_k\}_{k=0}^M\supset \ti{t}$ the estimate
\EQQ{
|B_{t'}(u,v)-B(u,v)|<\e
}
holds where
\EQQ{
B_{t'}(u,v)=\sum_{k=1}^M \LR{u(t'_{k-1}),v(t'_k)-v(t'_{k-1})}_{L^2_x}.
}
Put $t'=\{t_k\}_{k=0}^K \cup \{\ti{t}_k\}_{k=0}^N$.
Since $u(s)=u(t'_{n-1})$ on $s\in [t'_{n-1},t'_n)$, we have
\EQQ{
\LR{u(t'_{n-1}),v(t'_n)-v(t'_{n-1})}_{L^2_x}=\int_{t'_{n-1}}^{t'_n} \LR{u(s),v'(s)}_{L^2_x} ds
}
when $t_n'\neq \I$ and
\EQQ{
\LR{u(t'_{n-1}),v(t'_n)-v(t'_{n-1})}_{L^2_x}
&=\lim_{t\to\I}\LR{u(t'_{n-1}),v(t)-v(t'_{n-1})}_{L^2_x}-\lim_{t\to\I}\LR{u(t'_{n-1}),v(t)}_{L^2_x} \\
&=\int_{t'_{n-1}}^{t'_n} \LR{u(s),v'(s)}_{L^2_x} ds-\lim_{t\to\I}\LR{u(t),v(t)}_{L^2_x}
}
when $t_n' = \I$.
Thus, we conclude
\EQQ{
\Bigl|\int_{-\I}^{\I} \LR{u(s),v'(s)}_{L^2_x} ds-\lim_{t\to\I}\LR{u(t),v(t)}_{L^2_x}-B(u,v)\Bigr|<\e.
}
\end{proof}
Combining Remark \ref{V^2norm} and Proposition \ref{prop_Bform2}, we have the following corollary.
\begin{cor}  \label{V^2_A}
Let $A=S$ or $W_{\pm}$ and $v \in V^1_{-,\, A} \subset V^2_{-,\, A}$ be absolutely continuous on compact intervals.
Then, 
\EQQS{
 \|v\|_{V^2_A} \le \sup_{u \in U^2_A,\, \|u\|_{U^2_A}=1}
                       \Bigl|\int_{-\I}^{\I} \LR{u(t), A(t)\bigl(A(-\cdot )v\bigr)'(t)}_{L^2_x} dt 
                         - \lim_{t \to \I} \LR{u(t), v(t)}_{L^2_x}\Bigr|.
}
\end{cor}

%%%%%%%%%%%%%%%%%%%%%%%%%%%%%%%%%%%%%%%%%%%%%%%%%%%%%%%%
%%%%%%%%%%%%%%%%%%%%%%%%%%%%%%
%%%%%  Modulation estimate  %%%%%
%%%%%%%%%%%%%%%%%%%%%%%%%%%%%%
\begin{prop} \label{modulation}
We have 
\EQS{
&\| Q_M^S u\|_{L^2_{t,x}(\R^{1+d})}\lec M^{-1/2}\| u\|_{V^2_S},\ \ \ \ \ \| Q_{\ge M}^S u\|_{L^2_{t,x}(\R^{1+d})}\lec M^{-1/2}\| u\|_{V^2_S}, \label{mod} \\
  &\| Q_{<M}^S u\|_{V^2_S}\lec \| u\|_{V^2_S},\ \ \ \ \ \| Q_{\ge M}^S u\|_{V^2_S}\lec \| u\|_{V^2_S}, \notag \\
&\| Q_{<M}^S u\|_{U^2_S}\lec \| u\|_{U^2_S},\ \ \ \ \ \| Q_{\ge M}^S u\|_{U^2_S}\lec \| u\|_{U^2_S}. \notag 
}
The same estimates hold by replacing the Schr\"{o}dinger operator $S$ with the wave operators $W_{\pm}.$  
\end{prop}
For the proof of Proposition \ref{modulation}, see Corollary 2.18 in ~\cite{HHK}. 

The following lemma plays an important role to estimate the nonlinear terms.
The symbol $\ta+|\xi|^2$ (resp.~$\ta \pm |\xi|$) comes from the linear part of Schr\"odinger equation (resp.~the wave equation).
If we define $M$ as the left-hand side of (\ref{recover}), one derivative loss is recovered by Lemma \ref{Recovery} and \eqref{mod} in Proposition \ref{modulation}.

%%%%%%%%%%%%%%%%%%%%%%%%%%%%%%%%%%%%%%%%%%%%
%%%%%%%%%  Recover the derivative loss  %%%%%%%%%
%%%%%%%%%%%%%%%%%%%%%%%%%%%%%%%%%%%%%%%%%%%%
\begin{lem} \label{Recovery}
Let $\ta_3=\ta_1-\ta_2,\ \xi_3=\xi_1-\xi_2.$ 
If $|\xi_1| \gg \LR{\xi_2}$ or $\LR{\xi_1} \ll |\xi_2|,$ then it holds that 
\EQS{ \label{recover}
\max\big\{ \big|\ta_1+|\xi_1|^2\big|, \big|\ta_2+|\xi_2|^2\big|, \big|\ta_3 \pm |\xi_3| \big|  \big\} 
 \gec \max\{|\xi_1|^2, |\xi_2|^2\}.
} 
\end{lem}
%%%%%%%%%%%%%%%%%%%%%%%%%%%%%%%%%%%%%%%%%%%%%%%%%%%%%%%%
%%%%%%%%%%%%%%%%%  Proof of the lemma  %%%%%%%%%%%%%%%%%%%%
%%%%%%%%%%%%%%%%%%%%%%%%%%%%%%%%%%%%%%%%%%%%%%%%%%%%%%%%
\begin{proof}
We only prove the case of $|\xi_1| \gg \LR{\xi_2}$. By triangle inequality, $\ta_3=\ta_1-\ta_2$ and $\xi_3 = \xi_1 - \xi_2$, we have  
\EQS{ 
 (LHS\ of\ \eqref{recover}) &\gec \bigl|\ta_1 + |\xi_1|^2\bigr| + \bigl|\ta_2 + |\xi_2|^2\bigr| 
                                              + \bigl|\ta_3 \pm |\xi_3|\bigr| \notag \\
                                    &\ge  \bigl|\ta_1 + |\xi_1|^2 - (\ta_2 + |\xi_2|^2) - (\ta_3 \pm |\xi_3|)\bigr| \notag \\
                                    &=     \bigl| |\xi_1|^2 - |\xi_2|^2 \mp |\xi_1 - \xi_2| \bigr|.   \label{rec1}
}
Since $|\xi_1| \gg \LR{\xi_2}$, we see that $|\xi_1 - \xi_2| \sim |\xi_1|$. 
Hence 
\EQQS{ 
 \eqref{rec1}
                 \gec |\xi_1|^2.
}
\end{proof}
%%%%%%%%%%%%%%%%%%%%%%%%%%%%%%%%%%%%%%%%%%%%%%%%%%%%%%%%
%%%%%%%%%%%%%%%%%%%%%%%%%%%%%%%%%%%%%%
%%%%%  Definition of Duhamel terms  %%%%%%
%%%%%%%%%%%%%%%%%%%%%%%%%%%%%%%%%%%%%%  
We define the Duhamel terms as follows.
\begin{defn}
\EQS{
 &I_{T,S}(n,v)(t) :=-i/2\int_0^t\1_{[0,\, T]}(t')S(t-t') n(t') v(t')\, dt', \label{DuS} \\
 &I_{T,\, W_{\pm}}(u,v)(t) := \pm \int_0^t \1_{[0,T]}(t')W_{\pm}(t-t')\, \om  \big(u(t') \bar{v}(t')\big)\, dt' \label{DuW}
}
where $\om = (-\laplacian)^{1/2}.$
\end{defn}
The following statement is the Strichartz estimate for the Schr\"{o}dinger equation. 
%%%%%%%%%%%%%%%%%%%%%%%%%%%%%%%%%%%%%%%%%%%%%%%%%%
%%%%%  Strichartz estimate (Included endpoint case) %%%%%
%%%%%%%%%%%%%%%%%%%%%%%%%%%%%%%%%%%%%%%%%%%%%%%%%%
\begin{prop}  \label{Strichartz}
Let $d\ge 3$ and $(p_1, q_1),\ (p_2, q_2)$ satisfy 
$2\le q_i \le 2d/(d-2)$ and $2/p_i=d(1/2-1/q_i)$ for $i=1,2$. 
$\ p_2',\ q_2'$ satisfy 
$1/p_2+1/p_2'=1,\ 1/q_2+1/q_2'=1$. 
Then, it holds that
\EQS{
 &\| S(t)f\|_{L^{p_i}_t L^{q_i}_x} \lec \| f\|_{L^2_x}, \ \ \ i=1, 2, \label{Str1} \\
 &\Bigl{\|}\int_{0}^t S(t-t') g(t')dt'\Bigr{\|}_{L^{p_1}_t L^{q_1}_x} \lec \|g\|_{L^{p_2'}_t L^{q_2'}_x}.   \label{Str2}
}
Moreover, by duality, we have 
\EQQS{
 \|I_{T,S}(n,v)(t)\|_{L^{p_1}_t L^{q_1}_x} \lec \sup_{\|u\|_{L^{p_2}_t L^{q_2}_x}=1} 
                                                       \Bigl|\int_{\R} \int_{\R^d} \1_{[0,T]}nv\bar{u} dxdt \Bigr|.
}
\end{prop}
For the proofs of \eqref{Str1} and \eqref{Str2}, see \cite{Yajima}, ~\cite{GV3} and ~\cite{KeelTao}.
%%%%%%%%%%%%%%%%%%%%%%%%%%%%%%%%%%%%%%%%%%%%%%%%%%%%%%
%%%%%  Strichartz estimate for U^p type Bourgain space  %%%%%
%%%%%%%%%%%%%%%%%%%%%%%%%%%%%%%%%%%%%%%%%%%%%%%%%%%%%%
\begin{prop}\label{multilinear}
Let $T_0:\, L^2_x\times \cdots \times L^2_x\to L^1_{loc}(\R^d;\C)\ $ 
be a n-linear operator.
Assume that for some $ 1\le p<\infty$ and $1\le  q \le \I$, it holds that
$$\|T_0( S(\cdot)\phi_1,\ldots ,S(\cdot)\phi_n)\|_{L^p_t(\R;L^q_x(\R^d))}
   \lec \ds \prod_{i=1}^n\|\phi_i\|_{L^2_x}.$$
Then, there exists $T:U_S^p\times \cdots \times U_S^p\to 
      L^p_t(\R;L^q_x(\R^d))$ satisfying  
$$\|T(u_1,\ldots ,u_n)\|_{L^p_t(\R;L^q_x(\R^d))}\lec \ds \prod_{i=1}^n
                    \|u_i\|_{U^p_S},$$
such that
$T(u_1,\ldots,u_n)(t)(x)=T_0(u_1(t),\ldots ,u_n(t))(x)$
a.e. 
\end{prop}
See Proposition 2.19 in \cite{HHK} for the proof of the above proposition.
%%%%%%%%%%%%%%%%%%%%%%%%%%%%%%%%%%%%%%%%%%%%%%%%%%%%%
%%%%  Strichartz estimate for Schr\"{o}dinger equation  %%%%%
%%%%%%%%%%%%%%%%%%%%%%%%%%%%%%%%%%%%%%%%%%%%%%%%%%%%%
Combining Propositions \ref{embedding}, \ref{Strichartz} and \ref{multilinear}, we have the following corollary.
\begin{cor} \label{Strichartz-S}
Let $(p_1,q_1)$ satisfy the assumption in Proposition \ref{Strichartz} and $p \le p_1$. Then,
$U^p_S$ is continuously embedded in $L^{p_1}_tL^{q_1}_x$.
\end{cor}
\begin{prop} \label{unique}
Let $d \ge 4,\ k=(d-3)/2$ and $l=(d-4)/2$. \\
(i) Let $T>0$ and $u \in X^k_S([0,T]),\ u(0)=0$. Then, for any $\e>0$, there exists $0 \le T' \le T$ such that 
$\|u\|_{X^k_S([0,T'])} < \e $.    \\
(ii) Let $T>0$ and $u \in Y^k_S([0,T]),\ u(0)=0$. Then, for any $\e>0$, there exists $0 \le T' \le T$ such that 
$\|u\|_{Y^k_S([0,T'])} < \e $.   \\
(iii) Let $T>0$ and $n \in \dot{Y}^l_{W_{\pm}}([0,T])$, (resp. $Y^l_{W_{\pm}}([0,T])$), $n(0)=0$. 
Then, for any $\e>0$, there exists $0 \le T' \le T$ such that $\|n\|_{\dot{Y}^l_{W_{\pm}}([0,T'])}$ (resp. $\|n\|_{Y^l_{W_{\pm}}([0,T'])}$) $< \e $. 
\end{prop}
\begin{proof}
For the proofs of $(ii)$ and $(iii)$, see Proposition 2.24 in \cite{HHK}. 
For the proof of $(i)$,  
we only see that $\|u\|_{{E}^k([0,T'])}<\e $, which follows from $\|u\|_{{E}^k([0,T])}<\infty$.
\end{proof}
\begin{lem} \label{Mhigh}
  If $f, g$ are measurable functions, then
\EQ{\label{QQ}
\int_{\R} \int_{\R^d} f(t,x) \ol{Q_{\ge M}^S g(t,x)} dx dt=
\int_{\R} \int_{\R^d} \Big(Q_{\ge M}^S f(t,x)\Big) \ol{ g(t,x)} dx dt.}
\end{lem}
\begin{proof}
From the definition of $Q_{\ge M}^S$, we obtain 
\EQS{
 \F_x[Q_{\ge M}^S g](t, \xi ) &= \sum_{N \ge M}\F_x[S(\cdot )Q_N S(-\cdot )g](t, \xi ) \notag \\
                                     &= \sum_{2^n \ge M}e^{-it|\xi |^2}\F_x\bigl[\F_t^{-1}[\phi _n(\ta ) 
                                           \F_t[S(-\cdot )g](\ta )]\bigr](t, \xi ) \notag \\
                                     &= \sum_{2^n \ge M}e^{-it|\xi |^2} \F_t^{-1}\bigl[\phi _n(\ta ) 
                                           \F_t\bigl[e^{i \cdot |\xi |^2} \F_x[g]\bigr](\ta )\bigr](t, \xi ) \notag \\
                                     &= \sum_{2^n \ge M}e^{-it|\xi |^2}(\F_t^{-1}[\phi _n] \ast_{(t)} 
                                           e^{i \cdot |\xi |^2} \F_x[g])(t, \xi )  
                                             \label{a}
}
Applying the Plancherel theorem and \eqref{a}, we obtain that the left-hand side of \eqref{QQ} is equal to
\EQQ{
  &\int_{\R}\int_{\R^d}\F_x[f](t, \xi ) \ol{\F_x [Q_{\ge M}^S g](t, \xi )} d\xi dt  \notag \\ 
      =& \sum_{2^n \ge M} \int_{\R}\int_{\R^d}\int_{\R}e^{i(t-t')|\xi |^2} \F_x[f](t, \xi ) \ol{\F_t^{-1}[\phi _n](t-t')} 
           \ol{\F_x[g](t', \xi )} dt' d\xi dt\\
=&\sum_{2^n \ge M} \int_{\R}\int_{\R^d}\int_{\R}e^{i(t-t')|\xi |^2} \F_x[f](t, \xi ) \F_t^{-1}[\phi _n](t'-t) 
           \ol{\F_x[g](t', \xi )} dt' d\xi dt.
}
In the last line, we used $\ol{\F_t^{-1}[\phi _n](t-t')} = \F_t^{-1}[\phi _n](t'-t)$, which holds because $\phi _n $ is real valued.
Applying the Plancherel theorem and \eqref{a}, we obtain that the right-hand side of \eqref{QQ} is equal to
\EQQ{
& \int_{\R}\int_{\R^d} \F_x [Q_{\ge M}^S f](t, \xi ) \ol{\F_x[g](t, \xi )}d\xi dt  \notag \\
     =& \sum_{2^n \ge M} \int_{\R}\int_{\R^d} e^{-it|\xi |^2}(\F_t^{-1}[\phi _n] \ast_{(t)} e^{i \cdot |\xi |^2} \F_x[f])(t, \xi )
            \ol{\F_x[g](t, \xi )}d\xi dt     \notag \\ 
     =& \sum_{2^n \ge M} \int_{\R}\int_{\R^d}\int_{\R} e^{i(t'-t)|\xi |^2} \F_x[f](t', \xi )\F_t^{-1}[\phi _n](t-t') 
            \ol{\F_x[g](t, \xi )}dt' d\xi dt.
}
Thus, we conclude \eqref{QQ}.
\end{proof}
%
%
% 
%
%
%
%%%%%%%%%%%%%%%%%%%%%%%%%%%%
%%%%%  Trilinear estimate  %%%%%
%%%%%%%%%%%%%%%%%%%%%%%%%%%%
\begin{lem} \label{trilinear}
Let $d \ge 4,\ k=(d-3)/2,\ l=(d-4)/2,\ f_{N_3}:=P_{N_3} f,\ g_{N_2}:=P_{N_2} g$ and $h_{N_1}:=P_{N_1} h.$ Then, the following estimates hold: \\
(i) 
\vspace{-1.3\baselineskip}
\EQQS{ 
\Bigl| \int_{\R} \int_{\R^d}f_{N_3}g_{N_2}\ol{h_{N_1}}dxdt \Bigr| 
    \lec N_3^l \| f_{N_3}\|_{V^2_{W_{\pm}}}\| g_{N_2}\|_E\| h_{N_1}\|_E, 
}
(iia)
\vspace{-1.3\baselineskip}
\EQQS{
\Bigl| \int_{\R} \int_{\R^d} \Bigl( \sum_{N_3\ll N_2}f_{N_3}\Bigr) g_{N_2}\ol{h_{N_1}}dxdt \Bigr| 
    \lec \| f\|_{\dot{Y}^l_{W_{\pm}}}\| g_{N_2}\|_E\| h_{N_1}\|_E, 
}
(iib)
\vspace{-1.3\baselineskip}
\EQQS{
\Bigl| \int_{\R} \int_{\R^d} \Bigl( \sum_{N_3\ll N_2}f_{N_3}\Bigr) g_{N_2}\ol{h_{N_1}}dxdt \Bigr| 
    \lec \| f\|_{Y^l_{W_{\pm}}}\| g_{N_2}\|_E\| h_{N_1}\|_E, 
}
(iii) If $N_1 \sim N_3, N_3 > 2^2,\ M=\e N_3^2$ and $\e >0$ is small, then
\EQQS{
    \Bigl| \int_{\R} \int_{\R^d}(Q^{W_{\pm}}_{\ge M}f_{N_3}) \Bigl( \sum_{N_2 \ll N_3}g_{N_2}\Bigr) \ol{h_{N_1}} dxdt \Bigr| 
     \lec N_3^{-1/2}\| f_{N_3}\|_{V^2_{W_{\pm}}}\| g\|_{Y^k_S}\| h_{N_1}\|_E,
}
(iv) If $N_1 \sim N_3, N_3 > 2^2,\ M=\e N_3^2$ and $\e >0$ is small, then
\EQQS{
     \Bigl| \int_{\R} \int_{\R^d}(Q^{W_{\pm}}_{<M} f_{N_3}) \Bigl( \sum_{N_2 \ll N_3}g_{N_2}\Bigr) \ol{(Q^S_{\ge M} h_{N_1})}
               dxdt \Bigr|
      \lec N_3^{-1/2}\| f_{N_3}\|_{V^2_{W_{\pm}}}\| g\|_{E^k}\| h_{N_1}\|_{V^2_S},
}         
(v) If $N_1 \sim N_3, N_3 > 2^2,\ M=\e N_3^2$ and $\e >0$ is small, then
\EQQS{
     \Bigl| \int_{\R} \int_{\R^d} (Q^{W_{\pm}}_{<M} f_{N_3})\Bigl(\sum_{N_2 \ll N_3}Q^S_{\ge M}g_{N_2}\Bigr)\ol{h_{N_1}}
               dxdt \Bigr|
      \lec N_3^{-1/2}\| f_{N_3}\|_{V^2_{W_{\pm}}}\| g\|_{Y^k_S}\| h_{N_1}\|_E.
}                                                
(vi) If $N_1\sim N_3, N_3 > 2^2,\ M=\e N_3^2$ and $\e >0$ is small, then
\EQQS{
     \Bigl| \int_{\R} \int_{\R^d} (Q^{W_{\pm}}_{<M} f_{N_3}) \Bigl(\sum_{N_2 \ll N_3}Q^S_{\ge M}g_{N_2}\Bigr)
               \ol{(Q^S_{\ge M}h_{N_1})}dxdt \Bigr|
       \lec N_3^{-1/2}\| f_{N_3}\|_{V^2_{W_{\pm}}}\| g\|_{Y^k_S}\| h_{N_1}\|_{V^2_S}.
}
Here, the implicit constants may depend on $\e$. Moreover, $(i)$--$(vi)$ are
valid if $f_{N_3}$, $Q^{W_{\pm}}_{\ge M}f_{N_3}$, $Q^{W_{\pm}}_{<M} f_{N_3}$ in the left-hand sides are replaced by their complex conjugate.
\end{lem}
%%%%%%%%%%%%%%%%%%%%%%%%%%%%%%%%%%%%%%%%%%%%%%%%%%%%%%%%%%%%%%%%%%%%%%%%%%%%%%%%%%%%%%%%%%%%%%%%%%%%%%%%
%%%%%%%%%%%%%%%%%%%%%%%%%%%%%%%%%%%%%%%%%  Proof of Trilinear estimate  %%%%%%%%%%%%%%%%%%%%%%%%%%%%%%%%%%%%%%
%%%%%%%%%%%%%%%%%%%%%%%%%%%%%%%%%%%%%%%%%%%%%%%%%%%%%%%%%%%%%%%%%%%%%%%%%%%%%%%%%%%%%%%%%%%%%%%%%%%%%%%%
\begin{proof}
%%%%%%%%%%%%%%%%%%%%%%%%%%%%%%%%%%%%%%%%%%%%%%%%%%%%%%%%%%%%%%%%%%%%%%%%%%%%%%%%%%%%%%%%%%%%%%%%%%%
%%%%%%%%%%%%%%%%%%%%%%%%%%%%%%%%%%%%%%%%%  Proof of (i)  %%%%%%%%%%%%%%%%%%%%%%%%%%%%%%%%%%%%%%%%%%%%%%
%%%%%%%%%%%%%%%%%%%%%%%%%%%%%%%%%%%%%%%%%%%%%%%%%%%%%%%%%%%%%%%%%%%%%%%%%%%%%%%%%%%%%%%%%%%%%%%%%%%
First, we show $(i)$. 
By the H\"{o}lder inequality, we have 
\EQS{ \label{i1} 
 (\text{LHS\ of} \ (i)) 
   \lec \| f_{N_3}\|_{L^{\I}_tL^{d/2}_x}\| g_{N_2}\|_{L^2_tL^{2d/(d-2)}_x}\| h_{N_1}\|_{L^2_tL^{2d/(d-2)}_x}. 
}
The Sobolev inequality and Remark \ref{embed}
gives 
\EQS{ \label{i2} 
 \| f_{N_3}\|_{L^{\I}_tL^{d/2}_x}
   \lec \bigl{\|} |\na _x|^{(d-4)/2}f_{N_3}\bigr{\|}_{L^{\I}_tL^2_x} 
   \lec N_3^{(d-4)/2}\| f_{N_3}\|_{V^2_{W_{\pm}}}.
}  
Hence, from \eqref{i1} and \eqref{i2}, we obtain $(i)$.
By Remarks \ref{embed} and \ref{estx},
\EQQS{
\Bigl{\|} |\nabla_x|^{(d-4)/2} \sum_{N_3\ll N_2}f_{N_3} \Bigr{\|}_{L^{\I}_tL^2_x}\lec
\Big\| \sum_{N_3\ll N_2}f_{N_3} \Big\|_{\dot{Y}^{(d-4)/2}_{W_{\pm}}} 
\lec \| f \|_{\dot{Y}^{(d-4)/2}_{W_{\pm}}}.
}
Thus, we obtain $(iia)$ in the same manner as $(i)$.
By Remarks \ref{embed} and Lemma \ref{estX},
\EQQS{
\Bigl{\|} |\nabla_x|^{(d-4)/2} \sum_{1 \le N_3\ll N_2}f_{N_3} \Bigr{\|}_{L^{\I}_tL^2_x}\lec
\Big\| \sum_{1\le N_3\ll N_2}f_{N_3} \Big\|_{{Y}^{(d-4)/2}_{W_{\pm}}} 
\lec \| f \|_{{Y}^{(d-4)/2}_{W_{\pm}}},
}
\EQQS{
\Bigl{\|} |\nabla_x|^{(d-4)/2} \sum_{N_3<1} f_{N_3} \Bigr{\|}_{L^{\I}_tL^2_x}\lec
\Big\| \sum_{N_3<1}f_{N_3} \Big\|_{V^2_{W_{\pm}}} 
\lec \| P_{<1} f \|_{V^2_{W_{\pm}}}.
}
Thus, we obtain $(iib)$ in the same manner as $(i)$.
%%%%%%%%%%%%%%%%%%%%%%%%%%%%%%%%%%%%%%%%%%
%%%%%%%%%%%%%%%%%%%%%%%%%%%%%%%%%%%%%%%%%%
Next, we show $(iii)$.
By the H\"{o}lder  inequality, the Sobolev inequality and Proposition \ref{modulation}, we have   
\EQ{\label{e-iiia}
(\text{LHS\ of}\ (iii))  
      &\lec \bigl{\|}Q^{W_{\pm}}_{\ge M}f_{N_3}\bigr{\|}_{L^2_{t,x}} \Bigl{\|} \sum_{N_2 \ll N_3}g_{N_2}\Bigr{\|}_{L^{\I}_tL^d_x}
              \|h_{N_1}\|_{L^2_tL^{2d/(d-2)}_x}\\
&\lec N_3^{-1}\|f_{N_3}\|_{V^2_{W_{\pm}}} \Bigl{\|} |\na _x|^{(d-2)/2}\sum_{N_2 \ll N_3}g_{N_2}\Bigr{\|}_{L^{\I}_tL^2_x} \|h_{N_1}\|_{E}.
}
By Remark \ref{embed}, we have
\EQ{\label{e-iiib}
 \Bigl{\|} |\na _x|^{(d-2)/2} \sum_{N_2 <1} g_{N_2}\Bigr{\|}_{L^{\I}_tL^2_x} 
  \lec \|  P_{<1}g \|_{L^{\I}_tL^2_x}
  \lec \|P_{<1}g\|_{V^2_S}
  \lec \|g\|_{Y^k_S}.    
}
By $L^2_x$ orthogonality and Remark \ref{embed}, we have 
\EQ{\label{e-iiic}
 \Bigl{\|} |\na _x|^{(d-2)/2} \sum_{1\le N_2 \ll N_3} g_{N_2}\Bigr{\|}_{L^{\I}_tL^2_x}\lec
   \Bigl(\sum_{1 \le N_2 \ll N_3} N_2^{d-2} \|g_{N_2}\|_{V^2_S}^2\Bigr)^{1/2}
  \lec N_3^{1/2}\|g\|_{Y^k_S}. 
}
Collecting \eqref{e-iiia}--\eqref{e-iiic}, we obtain $(iii)$.
%%%%%%%%%%%%%%%%%%%%%%%%%%%%%%%%%%%%%%%%%%%%%%%%%%%%%%%%%%%%%%%%%%%%%%%%%%%%%%%%%%%%%%%%%%%%%%%%%%%
%%%%%%%%%%%%%%%%%%%%%%%%%%%%%%%%%%%%%%%%%%  Proof of (iv)  %%%%%%%%%%%%%%%%%%%%%%%%%%%%%%%%%%%%%%%%%%%%
%%%%%%%%%%%%%%%%%%%%%%%%%%%%%%%%%%%%%%%%%%%%%%%%%%%%%%%%%%%%%%%%%%%%%%%%%%%%%%%%%%%%%%%%%%%%%%%%%%%
Next, we show $(iv)$.
Applying the H\"{o}lder inequality, we have 
\EQS{ \label{iv}
 (\text{LHS of}\ (iv)) 
      &\lec \|Q_{<M}^{W_{\pm}}f_{N_3}\|_{L^{\I}_tL^2_x}\Bigl{\|}\sum_{N_2 \ll N_3}g_{N_2}\Bigr{\|}_{L^2_tL^{\I}_x}
                \|Q_{\ge M}^S h_{N_1}\|_{L^2_{t,x}}. 
}
By Remark \ref{embed} and Proposition \ref{modulation}, we have 
\EQS{
& \|Q_{<M}^{W_{\pm}} f_{N_3}\|_{L^{\I}_tL^2_x} \lec \|Q_{<M}^{W_{\pm}} f_{N_3}\|_{V^2_{W_{\pm}}}
                                                      \lec \|f_{N_3}\|_{V^2_{W_{\pm}}},\label{e-iva}\\
&\|Q_{\ge M}^S h_{N_1}\|_{L^2_{t,x}} \lec N_1^{-1}\|h_{N_1}\|_{V^2_S}.\label{e-ivb}
} 
By the triangle inequality and the Bernstein inequality (see e.g. (A.6) on page 333 in $\!$~\cite{Tao}), we have 
\EQ{ \label{g}
  \Bigl{\|} \sum_{N_2 \ll N_3}g_{N_2}\Bigr{\|}_{L^2_tL^{\I}_x}
   \lec \sum_{N_2 \ll N_3} \|g_{N_2}\|_{L^2_tL^{\I}_x}
 \lec \sum_{N_2 \ll N_3} N_2^{(d-2)/2}\|g_{N_2}\|_{E}
}
Since $(d-2)/2 >0$, by Mihlin-H\"{o}rmander's multiplier theorem, we have  
\EQ{ \label{iv2}
\sum_{N_2 <1} N_2^{(d-2)/2}\|g_{N_2}\|_{E}
\lec  \sum_{N_2 <1} N_2^{(d-2)/2} \| P_{N_2}P_{<1} g\|_E \lec \|P_{<1}g\|_E. 
}
By the Cauchy-Schwarz inequality, we have 
\EQ{ \label{iv3}
\sum_{1 \le N_2 \ll N_3} N_2^{(d-2)/2}{\|}g_{N_2}{\|}_{E}
   &\lec \Bigl(\sum_{1 \le N_2 \ll N_3}N_2\Bigr)^{1/2} \Bigl( \sum_{1 \le N_2 \ll N_3} N_2^{d-3}\|g_{N_2}\|_E^2 
             \Bigr)^{1/2} \\
   &\lec N_3^{1/2} \|g\|_{E^k}.
}
Collecting \eqref{iv}--\eqref{iv3} and $N_1 \sim N_3$, we obtain $(iv)$.
%%%%%%%%%%%%%%%%%%%%%%%%%%%%%%%%%%%%%%%%%%%%%%%%%%%%%%%%%%%%%%%%%%%%%%%%%%%%%%%%%%%%%%%%%%%%%%%%%%%
%%%%%%%%%%%%%%%%%%%%%%%%%%%%%%%%%%%%%%%%%%  Proof of (v)  %%%%%%%%%%%%%%%%%%%%%%%%%%%%%%%%%%%%%%%%%%%%
%%%%%%%%%%%%%%%%%%%%%%%%%%%%%%%%%%%%%%%%%%%%%%%%%%%%%%%%%%%%%%%%%%%%%%%%%%%%%%%%%%%%%%%%%%%%%%%%%%%
Next, we show $(v)$.
Applying the H\"{o}lder inequality, the Sobolev inequality and \eqref{e-iva}, we have 
\EQ{ \label{v}
(\text{LHS\ of}\ (v)) 
      &\lec \|Q_{<M}^{W_{\pm}}f_{N_3}\|_{L^{\I}_tL^2_x}\Bigl{\|}\sum_{N_2 \ll N_3}
                Q^S_{\ge M}\, g_{N_2}\Bigr{\|}_{L^2_tL^d_x}\| h_{N_1}\|_{L^2_tL^{2d/(d-2)}_x}\\
      &\lec \|f_{N_3}\|_{V^2_{W_{\pm}}} \Bigl{\|}|\na _x|^{(d-2)/2}\sum_{N_2 \ll N_3}
                Q^S_{\ge M}\, g_{N_2}\Bigr{\|}_{L^2_{t,x}}\| h_{N_1}\|_E.
}
By Proposition \ref{modulation}, we have 
\EQS{ \label{vg2}
 \Bigl{\|} |\na _x|^{(d-2)/2}\sum_{N_2 <1}Q^S_{\ge M} g_{N_2}\Bigr{\|}_{L^2_{t,x}}
  \lec \|Q^S_{\ge M}P_{<1} g\|_{L^2_{t,x}} 
  \lec N_3^{-1}\|P_{<1}g\|_{V^2_S}.
}
By $L^2_x$ orthogonality and Proposition \ref{modulation}, we have  
\EQ{ \label{vg3}
\Bigl{\|} |\na _x|^{(d-2)/2}\sum_{1\le N_2 \ll  N_3}Q^S_{\ge M} g_{N_2} \Bigr{\|}_{L^2_{t,x}}
  &\lec \Bigl( \sum_{1 \le N_2 \ll N_3} \bigl{\|}|\na _x|^{(d-2)/2}
            Q^S_{\ge M}g_{N_2}\bigr{\|}_{L^2_{t,x}}^2\Bigr)^{1/2}\\
  &\lec \Bigl( \sum_{1 \le N_2 \ll N_3} N_2^{d-2}N_3^{-2} 
            \|g_{N_2}\|_{V^2_S}^2\Bigr)^{1/2}\\
  &\lec N_3^{-1/2}\|g\|_{Y^k_S}. 
}
From \eqref{v}--\eqref{vg3}, we obtain $(v)$.
%%%%%%%%%%%%%%%%%%%%%%%%%%%%%%%%%%%%%%%%%%%%%%%%%%%%%%%%%%%%%%%%%%%%%%%%%%%%%%%%%%%%%%%%%%%%%%%%%%%
%%%%%%%%%%%%%%%%%%%%%%%%%%%%%%%%%%%%%%%%%%  Proof of (vi)  %%%%%%%%%%%%%%%%%%%%%%%%%%%%%%%%%%%%%%%%%%%%
%%%%%%%%%%%%%%%%%%%%%%%%%%%%%%%%%%%%%%%%%%%%%%%%%%%%%%%%%%%%%%%%%%%%%%%%%%%%%%%%%%%%%%%%%%%%%%%%%%%
Finally, we show $(vi)$.
By the H\"{o}lder inequality, the triangle inequality, the Bernstein inequality, \eqref{e-iva} and \eqref{e-ivb}, we have
\EQ{ \label{vi}
(\text{LHS\ of}\ (vi)) 
      & \lec \| Q_{<M}^{W_{\pm}} f_{N_3}\|_{L^{\I}_tL^2_x}
             \Bigl{\|} \sum_{N_2 \ll N_3} Q_{\ge M}^S g_{N_2}\Bigr{\|}_{L^2_tL^{\I}_x}
              \|Q_{\ge M}^S h_{N_1}\|_{L^2_{t,x}}\\
&\lec \|  f_{N_3}\|_{V_{W_\pm}^2}
             \sum_{N_2 \ll N_3} N_2^{d/2} \| Q_{\ge M}^S g_{N_2} \|_{L^2_{t,x}}
              N_1^{-1}\| h_{N_1}\|_{V^2_S}.
}
By Proposition \ref{modulation}, we have  
\EQS{ \label{gn3}
 \sum_{N_2 <1} N_2^{d/2}\| Q_{\ge M}^S g_{N_2} \|_{L^2_{t,x}}
   \lec \| Q_{\ge M}^S P_{<1}g \|_{L^2_{t,x}}
   \lec N_3^{-1}\|P_{<1}g\|_{V^2_S}.
}
By the Cauchy-Schwarz inequality and Proposition \ref{modulation}, we obtain 
\EQ{ \label{gn4}
\sum_{1 \le N_2 \ll N_3} N_2^{d/2} \| Q_{\ge M}^S g_{N_2} \|_{L^2_{t,x}}
  &\lec \Big(\sum_{1 \le N_2 \ll N_3} N_2^3\Big)^{1/2}   \Bigl( \sum_{1 \le N_2 \ll N_3} N_2^{d-3} \| Q_{\ge M}^S g_{N_2} \|_{L^2_{t,x}}^2\Bigr)^{1/2}\\
  &\lec N_3^{3/2}\Bigl( \sum_{1 \le N_2 \ll N_3}N_2^{d-3} N_3^{-2}\|g_{N_2}\|_{V^2_S}^2\Bigr)^{1/2}\\
  &\lec N_3^{1/2} \|g\|_{Y^k_S}.
}
From \eqref{vi}--\eqref{gn4} and $N_1 \sim N_3$, we obtain $(vi)$.
We can easily check that the proofs of $(i)$ -- $(vi)$ above are valid if $f_{N_3}$, $Q^{W_{\pm}}_{\ge M}f_{N_3}$, $Q^{W_{\pm}}_{<M} f_{N_3}$ in the left-hand sides are replaced by their complex conjugate.
\end{proof}
%%%%%%%%%%%%%%%%%%%%%%%%%%%%%%%%%%%%%%%%%%%%%%%%%%%%%%%%%%%%%%%%%%%%%%%%%%%%%%%%%%%%%%%%%%%
%%%%%%%%%%%%%%%%%%%%%%%%%%%%%%%%%%%%%%%%%%%%%%%%%%%%%%%%%%%%%%%%%%%%%%%%%%%%%%%%%%%%%%%%%%%
%%%%%%%%%%%%%%%%%%%%%%%%%%%%%%%%%%%%%%%  Section 3  %%%%%%%%%%%%%%%%%%%%%%%%%%%%%%%%%%%%%%%%%
%%%%%%%%%%%%%%%%%%%%%%%%%%%%%%%%%%%%%%%%%%%%%%%%%%%%%%%%%%%%%%%%%%%%%%%%%%%%%%%%%%%%%%%%%%%
%%%%%%%%%%%%%%%%%%%%%%%%%%%%%%%%%%%%%%%%%%%%%%%%%%%%%%%%%%%%%%%%%%%%%%%%%%%%%%%%%%%%%%%%%%%

\section{bilinear estimates}
In this section, we give bilinear estimates for the Duhamel terms
\eqref{DuS} and \eqref{DuW}.   
%%%%%%%%%%%%%%%%%%%%%%%%%%%%%%%%%%%%%%%%%%%%%%%%%%%%%%%%%%%%%%%%%%%%%%%%%%%%%%%%%%%%%%%%%%%%%%%%%
%%%%%%%%%%%%%%%%%%%%%%%%%%%%%%%%%%%  Homogeneous case Wave %%%%%%%%%%%%%%%%%%%%%%%%%%%%%%%%%%%%%%
%%%%%%%%%%%%%%%%%%%%%%%%%%%%%%%%%%%%%%%%%%%%%%%%%%%%%%%%%%%%%%%%%%%%%%%%%%%%%%%%%%%%%%%%%%%%%%%%%
\begin{prop}\label{B.E.homo}
Let $d \ge 4,\ k=(d-3)/2$ and $l=(d-4)/2.$ Then for all $0<T<\I$,
it holds that
\EQS{
&\| I_{T,S}(n,v)\|_{X^k_S}\lec \| n\|_{\dot{Y}^l_{W_{\pm}}}\| v\|_{X^k_S}, \label{B.E.S.} \\
&\| I_{T,\, W_{\pm}}(u,v)\|_{\dot{Z}^l_{W_{\pm}}}\lec \|u\|_{X^k_S}\|v\|_{X^k_S}. \label{B.E.W.}
}
Here, the implicit constants do not depend on $T$.
\end{prop}
\begin{proof}
Let$\ u_{N_1}=P_{N_1}u,\  v_{N_2}=P_{N_2}v,\ n_{N_3}=P_{N_3}n.$
First, we prove \eqref{B.E.S.}. Since $\|\cdot\|_{X^k_S}=\|\cdot\|_{Y^k_S}+\|\cdot\|_{E^k}$, we need to show
\EQS{
    \| I_{T,S}(n,v)\|_{E^k} \lec \| n\|_{\dot{Y}^l_{W_{\pm}}}\| v\|_{X^k_S},\label{J-E}\\
    \| I_{T,S}(n,v)\|_{Y^k_S} \lec \| n\|_{\dot{Y}^l_{W_{\pm}}}\| v\|_{X^k_S}.\label{J-Y}
}
%%%%%%%%%%%%%%%%%%%%%%%%%%%%%%%%%%%%%%%%%%%%%%%%%%%%%%%%%%%%%%%%%%%%%%%%%%%%%%%%%%%%%%%%%%%%%%%%%%%
%%%%%%%%%%%%%%%%%%%%%%%%%%%%%%%%%%%%% Estimate of \dot{E}^k norm %%%%%%%%%%%%%%%%%%%%%%%%%%%%%%%%%%%%%%
%%%%%%%%%%%%%%%%%%%%%%%%%%%%%%%%%%%%%%%%%%%%%%%%%%%%%%%%%%%%%%%%%%%%%%%%%%%%%%%%%%%%%%%%%%%%%%%%%%%
By the definition of $E^k$ norm, we have  
\EQ{
 (\text{LHS\ of}\ (\ref{J-E}))^2 \lec \|P_{<1}I_{T,S}(n,v)\|_E^2 + \sum_{N_1 \ge 1}N_1^{d-3}\| P_{N_1}I_{T,S}(n,v)\|_E^2.\label{J-E1}
}
Put
\EQQS{   
 J_{1,E} &:= \sum_{N_1 \ge 1} N_1^{d-3} \Bigl{\|} \int_0^t \1_{[0,T]}(t')
                 S(t-t') \sum_{N_2 \sim N_1} \sum_{N_3 \ll N_2} P_{N_1} \bigl( n_{N_3}(t') v_{N_2}(t')\bigr) dt' \Bigr{\|}_E^2, \\
 J_{2,E} &:= \sum_{N_1 \ge 1} N_1^{d-3} \Bigl{\|} \int_0^t \1_{[0,T]}(t')
                 S(t-t') \sum_{N_2 \gec N_1} \sum_{N_3 \sim N_2} P_{N_1} \bigl( n_{N_3}(t') v_{N_2}(t')\bigr) dt' \Bigr{\|}_E^2, \\
 J_{3,E} &:= \sum_{N_1 \ge 1} N_1^{d-3} \Bigl{\|} \int_0^t \1_{[0,T]}(t')
                 S(t-t') \sum_{N_2 \ll N_1} \sum_{N_3 \sim N_1} P_{N_1} \bigl( n_{N_3}(t') v_{N_2}(t')\bigr) dt' \Bigr{\|}_E^2. 
}
We will prove $J_{i,E}\lec \|n\|_{\dot{Y}_{W_{\pm}}^l}^2\|v\|_{X^k_S}^2$ for $i=1,2,3$ below.
%%%%%%%%%%%%%%%%%%%%%%%%%%%
%%%%% Estimate of J_{1,E} %%%%%
%%%%%%%%%%%%%%%%%%%%%%%%%%%
By Proposition \ref{Strichartz} and Lemma \ref{trilinear} $(iia)$, we have 
\EQQ{
 J_{1,E} &\lec \sum_{N_1 \ge 1} N_1^{d-3}\sup_{\|u\|_E =1}\Bigl| \sum_{N_2 \sim N_1}\sum_{N_3 \ll N_2} 
                     \int_{\R}\int_{\R^d} \1_{[0,T]}n_{N_3}v_{N_2}\ol{u_{N_1}} dxdt \Bigr|^2\\
&\lec \|n\|_{\dot{Y}_{W_{\pm}}^l}^2\sum_{N_1 \ge 1}\sum_{N_2 \sim N_1} N_1^{d-3}
  \|v_{N_2}\|_E^2\sup_{\|u\|_E =1}\|u_{N_1}\|_E^2.\\
}
Since $\sup_{\|u\|_E =1} \|u_{N_1}\|_E \lec 1$, we obtain
\EQQ{
 J_{1,E} \lec \|n\|_{\dot{Y}_{W_{\pm}}^l}^2\sum_{N_2 \gec 1} N_2^{d-3}\|v_{N_2}\|_E^2
            \lec \|n\|_{\dot{Y}_{W_{\pm}}^l}^2\|v\|_{X^k_S}^2. 
}
%%%%%%%%%%%%%%%%%%%%%%%%%%%
%%%%% Estimate of J_{2,E} %%%%%
%%%%%%%%%%%%%%%%%%%%%%%%%%%
By the triangle inequality, Proposition \ref{Strichartz} and Lemma \ref{trilinear} $(i)$, we have
\EQQ{
 J_{2,E}  &\lec  \sum_{N_1\ge 1} N_1^{d-3} \Bigl( \sum_{N_2 \gec N_1} \sum_{N_3\sim N_2} \Bigl{\|}
                             \int_0^t \1_{[0,T]}(t') S(t-t') P_{N_1} \bigl( n_{N_3}(t')v_{N_2}(t') \bigr)dt'
                               \Bigr{\|}_E\Bigr)^{2}\\
                  &\lec \sum_{N_1 \ge 1} N_1^{d-3} \Bigl( \sum_{N_2 \gec N_1}\sum_{N_3\sim N_2} \sup_{\|u\|_E=1} 
                           \Bigl| \int_{\R} \int_{\R^d} \1_{[0,T]}n_{N_3}v_{N_2}\ol{u_{N_1}}dxdt \Bigr|\Bigr)^{2}\\
&\lec \sum_{N_1 \ge 1}  \Bigl( \sum_{N_2 \gec N_1}\sum_{N_3\sim N_2} N_1^{(d-3)/2}N_3^{(d-4)/2}\, \|v_{N_2}\|_E\|n_{N_3}\|_{V^2_{W_{\pm}}}\Bigr)^{2}.
}
Since $\|\cdot\|_{\ell^2\ell^1}\le \|\cdot\|_{\ell^1\ell^2}$, by the Cauchy-Schwarz inequality, we obtain 
 \EQQ{
J_{2,E}^{1/2}  &\lec  \sum_{N_2 \gec 1}\sum_{N_3\sim N_2} \Bigl(
   \sum_{N_1 \lec N_2} N_1^{d-3} N_3^{d-4}\, \|v_{N_2}\|^2_E\|n_{N_3}\|^2_{V^2_{W_{\pm}}}\Bigr)^{1/2}\\
&\lec
 \sum_{N_2 \gec 1}\sum_{N_3\sim N_2}N_2^{(d-3)/2}N_3^{(d-4)/2}\, \|v_{N_2}\|_E\|n_{N_3}\|_{V^2_{W_{\pm}}}\\
&\lec \|n\|_{\dot{Y}_{W_{\pm}}^l}\|v\|_{X^k_S}.
}  
%%%%%%%%%%%%%%%%%%%%%%%%%%%
%%%%% Estimate of J_{3,E} %%%%%
%%%%%%%%%%%%%%%%%%%%%%%%%%%
Next, we consider the estimate of $J_{3,E}$. We take $M=\e N_1^2$ for sufficiently small $\e >0$. 
Then, from Lemma \ref{Recovery}, we have  
\EQQS{
   &P_{N_1}Q_{<M}^S\bigl( (Q_{<M}^{W_{\pm}}n_{N_3})(Q_{<M}^S v_{N_2}) \bigr) \\
   &= P_{N_1}Q_{<M}^S \Bigl[ \F^{-1}\Bigl( \int_{\ta_1 = \ta_2 + \ta_3,\, \xi_1 = \xi_2 + \xi_3} 
       \widehat{(Q_{<M}^{W_{\pm}}n_{N_3})}(\ta_3, \xi_3) \widehat{(Q_{<M}^S v_{N_2})}(\ta_2, \xi_2) \Bigr) \Bigr]
   =0
}
when $N_1 \gg \LR{N_2}$.
Therefore, 
\EQQS{
 P_{N_1}(n_{N_3}v_{N_2}) = \sum_{i=1}^4 P_{N_1}F_i, 
}
where 
\EQQS{
 &F_1 := (Q_{\ge M}^{W_{\pm}}n_{N_3}) v_{N_2}, \qquad \qquad \ \ \,    
 F_2 := Q_{\ge M}^S\bigl((Q_{<M}^{W_{\pm}}n_{N_3}) v_{N_2}\bigr), \notag \\
 &F_3 := (Q_{<M}^{W_{\pm}}n_{N_3}) (Q_{\ge M}^S v_{N_2}),\qquad 
 F_4 := -Q_{\ge M}^S\bigl((Q_{<M}^{W_{\pm}}n_{N_3}) (Q_{\ge M}^S v_{N_2})\bigr).
}
For the estimate of $F_1$, 
we apply Proposition \ref{Strichartz} and Lemma \ref{trilinear} $(iii)$ to have
\EQQS{   
 &\sum_{N_1 \ge 1}N_1^{d-3}\Bigl{\|}\int_0^t\1_{[0,T]}(t')S(t-t')\sum_{N_2 \ll N_1}
    \sum_{N_3 \sim N_1} P_{N_1}F_1 dt' \Bigr{\|}_E^2  \\
 &\lec \sum_{N_1 \ge 1}N_1^{d-3}\sup_{\|u\|_E=1} \Bigl{|}\sum_{N_2 \ll N_1}\sum_{N_3 \sim N_1}
          \int_{\R} \int_{\R^d} \1_{[0,T]}(Q_{\ge M}^{W_{\pm}}n_{N_3}) v_{N_2} \ol{u_{N_1}}dxdt \Bigr{|}^2  \\
 &\lec \sum_{N_3 \gec 1} N_3^{d-3} \bigl(N_3^{-1/2}\|n_{N_3}\|_{V^2_{W_{\pm}}} \|v\|_{Y^k_S}\bigr)^2  \\
 &\lec \|n\|_{\dot{Y}^l_{W_{\pm}}}^2\|v\|_{X^k_S}^2.
} 
For the estimate of $F_2$,  
we apply Corollary \ref{Strichartz-S}, Corollary \ref{U^2_A}, Lemma \ref{Mhigh}, Lemma \ref{trilinear} $(iv)$ and    
\EQ{ \label{ct}
 \|\1_{[0,T]} u_{N_1}\|_{V^2_S} \lec \|u_{N_1}\|_{V^2_S} \lec \|u\|_{V^2_S} 
}
to have
\EQQS{
  &\sum_{N_1 \ge 1} N_1^{d-3}\Bigl{\|}\int_0^t\1_{[0,T]}(t')S(t-t')\sum_{N_2 \ll N_1}
    \sum_{N_3 \sim N_1} P_{N_1}F_2 dt' \Bigr{\|}_E^2 \\
  &\lec\sum_{N_1 \ge 1} N_1^{d-3}\Bigl{\|}\int_0^t\1_{[0,T]}(t')S(t-t')\sum_{N_2 \ll N_1}
    \sum_{N_3 \sim N_1} P_{N_1}F_2 dt' \Bigr{\|}_{U^2_S}^2 \\
  &\lec \sum_{N_1 \ge 1} N_1^{d-3}\sup_{\|u\|_{V^2_S}=1} \Bigl{|}\sum_{N_2 \ll N_1}\sum_{N_3 \sim N_1}
           \int_{\R} \int_{\R^d} \1_{[0,T]} \Big(Q_{\ge M}^S\bigl((Q_{<M}^{W_{\pm}}n_{N_3}) v_{N_2}\bigr)\Big) 
            \ol{u_{N_1}} dxdt \Bigr{|}^2 \\
  &\lec \sum_{N_3 \gec 1} N_3^{d-3} \bigl(N_3^{-1/2}\|n_{N_3}\|_{V^2_{W_{\pm}}} \|v\|_{E^k} \bigr)^2 \\ 
  &\lec \|n\|_{\dot{Y}^l_{W_{\pm}}}^2\|v\|_{X^k_S}^2.
}
For the estimate of $F_3$, 
we apply Proposition \ref{Strichartz} and Lemma \ref{trilinear} $(v)$ to have
\EQQS{
 &\sum_{N_1 \ge 1} N_1^{d-3}\Bigl{\|}\int_0^t\1_{[0,T]}(t')S(t-t')\sum_{N_2 \ll N_1}
    \sum_{N_3 \sim N_1} P_{N_1}F_3 dt' \Bigr{\|}_E^2 \\
 &\lec \sum_{N_1 \ge 1} N_1^{d-3}\sup_{\|u\|_E=1} \Bigl{|}\sum_{N_2 \ll N_1}
           \sum_{N_3 \sim N_1} \int_{\R} \int_{\R^d} \1_{[0,T]}(Q_{<M}^{W_{\pm}}n_{N_3}) (Q_{\ge M}^S v_{N_2})\ol{u_{N_1}}
            dxdt \Bigr{|}^2 \\
 &\lec \sum_{N_3 \gec 1} N_3^{d-3} \bigl(N_3^{-1/2} \|n_{N_3}\|_{V^2_{W_{\pm}}} \|v\|_{Y^k_S}\bigr)^2 \\
 &\lec \|n\|_{\dot{Y}^l_{W_{\pm}}}^2\|v\|_{X^k_S}^2.
}
For the estimate of $F_4$,  
we apply Corollary \ref{Strichartz-S}, Corollary \ref{U^2_A}, Lemma \ref{Mhigh}, Lemma \ref{trilinear} $(vi)$ and 
\eqref{ct} to have
\EQQS{
  &\sum_{N_1 \ge 1} N_1^{d-3}\Bigl{\|}\int_0^t\1_{[0,T]}(t')S(t-t')\sum_{N_2 \ll N_1}
    \sum_{N_3 \sim N_1} P_{N_1}F_4 dt' \Bigr{\|}_E^2 \\
  &\lec \sum_{N_1 \ge 1} N_1^{d-3}\Bigl{\|}\int_0^t\1_{[0,T]}(t')S(t-t')\sum_{N_2 \ll N_1}
    \sum_{N_3 \sim N_1} P_{N_1}F_4 dt' \Bigr{\|}_{U^2_S}^2 \\
  &\lec \sum_{N_1 \ge 1} N_1^{d-3}\sup_{\|u\|_{V^2_S}=1} \Bigl{|}\sum_{N_2 \ll N_1}\sum_{N_3 \sim N_1}
           \int_{\R} \int_{\R^d} \1_{[0,T]} \Bigl(Q_{\ge M}^S\bigl((Q_{<M}^{W_{\pm}}n_{N_3}) (Q_{\ge M}^S v_{N_2})\bigr)\Bigr)\ol{u_{N_1}}
            dxdt\Bigr{|}^2 \\
  &\lec \sum_{N_3 \gec 1} N_3^{d-3} \bigl(N_3^{-1/2} \|n_{N_3}\|_{V^2_{W_{\pm}}} \|v\|_{Y^k_S}\bigr)^2 \\
  &\lec \|n\|_{\dot{Y}^l_{W_{\pm}}}^2\|v\|_{X^k_S}^2.
}
Collecting the estimates of $F_1, F_2, F_3$ and $F_4$, we obtain $J_{3,E} \lec \|n\|_{\dot{Y}^l_{W_{\pm}}}^2\|v\|_{X^k_S}^2$. Thus,
\EQ{
\sum_{N_1 \ge 1}N_1^{d-3}\| P_{N_1}I_{T,S}(n,v)\|_E^2 \lec \|n\|_{\dot{Y}^l_{W_{\pm}}}^2\|v\|_{X^k_S}^2.\label{J-E2}
}
Note that we also have
\EQ{
\sum_{N_1 \ge 1}N_1^{d-3}\| P_{N_1}I_{T,S}(n,v)\|_{L_t^\infty L_x^2}^2 \lec \|n\|_{\dot{Y}^l_{W_{\pm}}}^2\|v\|_{X^k_S}^2\label{J-Y0}
}
in the same manner as the proof of \eqref{J-E2} since $(p_1,q_1)=(\infty,2)$ also  satisfies the assumption of Proposition \ref{Strichartz}.
%%%%%%%%%%%%%%%%%%%%%%%%%%%%%%%%%%%%%%%%%%%%%%%%%%%%%%%%%%%%%%%%%%%%%%%%%%%%%%%%%%%%%%%%%%%%
%%%%%%%%%%%%%%%%%%%%%%%%%%%%%%%%%%%%%  Low frequency  %%%%%%%%%%%%%%%%%%%%%%%%%%%%%%%%%%%%%%%%
%%%%%%%%%%%%%%%%%%%%%%%%%%%%%%%%%%%%%%%%%%%%%%%%%%%%%%%%%%%%%%%%%%%%%%%%%%%%%%%%%%%%%%%%%%%%
Next, we show
\EQ{
 \| P_{<1} I_{T,S} (n,v)\|_E \lec \|n\|_{\dot{Y}^l_{W_{\pm}}} \|v\|_{X^k_S}. \label{ihE1} 
}
%%%%%%%%%%%%%%%%%%%%%%%%%%%%%%%%%%%%%%%%%%%%%%%%%%%%%%%%%%%%%%%%%%%%%%%%%%%%%%%%%%%%%%%%%%%%
%%%%%%%%%%%%%%%%%%%%%%%%%%%%%%%  Schr\"{o}dinger  %%%%%%%%%%%%%%%%%%%%%%%%%%%%%%%%%%%%%%%%%%%%%
%%%%%%%%%%%%%%%%%%%%%%%%%%%%%%%%%%%%%%%%%%%%%%%%%%%%%%%%%%%%%%%%%%%%%%%%%%%%%%%%%%%%%%%%%%%%
In the same manner as the proof of Lemma \ref{trilinear} $(iia)$, we have
\EQQ{
\| n \|_{L_t^\infty L_x^{d/2}}\lec \Big\| |\nabla_x|^{(d-4)/2}\sum_{N} P_N n \Big\|_{L_t^\infty L_x^2} \lec \Big(\sum_{N} N^{2l}\|P_N n\|^2_{V^2_{W_\pm}}\Big)^{1/2}=\|n\|_{\dot{Y}^l_{W_{\pm}}}.
}
Thus, by Proposition \ref{Strichartz} and the H\"older inequality, the left-hand side of \eqref{ihE1} is bounded by
\EQ{ 
&\sup_{\|u\|_E=1}\Bigl| \int_{\R} \int_{\R^d} \1_{[0,T]}nv \ol{P_{<1}u} dxdt \Bigr|\\
\lec & \| n \|_{L_t^\infty L_x^{d/2}} \|v\|_E \sup_{\|u\|_E=1}\|P_{<1}u\|_E \lec \|n\|_{\dot{Y}^l_{W_{\pm}}} \|v\|_{E^k}.\label{ihE2}
}
Thus, we obtain \eqref{ihE1}.
From \eqref{J-E1}, \eqref{J-E2} and \eqref{ihE1},
we conclude \eqref{J-E}.
%%%%%%%%%%%%%%%%%%%%%%%%%%%%%%%%%%%%%%%%%%%%%%%%%%%%%%%%%%%%%%%%%%%%%%%%%%%%%%%%%%%%%%%%%%%%%%%%%%%%%%%%%%
%%%%%%%%%%%%%%%%%%%%%%%%%%%%%%%%%%%%%%%%%  Estimate of Y^k_S norm  %%%%%%%%%%%%%%%%%%%%%%%%%%%%%%%%%%%%%%%%%%
%%%%%%%%%%%%%%%%%%%%%%%%%%%%%%%%%%%%%%%%%%%%%%%%%%%%%%%%%%%%%%%%%%%%%%%%%%%%%%%%%%%%%%%%%%%%%%%%%%%%%%%%%%

Next, we prove \eqref{J-Y}. By the definition of $\|\cdot\|_{Y^k_S}$, we only need to show
\EQS{
&\sum_{N_1 \ge 1} N_1^{d-3}\|P_{N_1}I_{T,S}(n,v)\|^2_{V^2_S} \lec \| n\|^2_{\dot{Y}^l_{W_{\pm}}}\| v\|^2_{X^k_S},\label{J-Y1}\\
&\|P_{<1}I_{T,S}(n,v)\|^2_{V^2_S} \lec \| n\|^2_{\dot{Y}^l_{W_{\pm}}}\| v\|^2_{X^k_S}.\label{J-Y2}
}
By Corollary \ref{V^2_A} and Remark \ref{embed}, the left-hand side of \eqref{J-Y1} is bounded by
\EQQ{
 &\sum_{N_1 \ge 1} N_1^{d-3} \sup_{\|u\|_{U^2_S}=1}\Bigl| \int_{-\I}^{\I} \LR{u(t), 
           S(t)\bigl( S(-\cdot ) P_{N_1}I_{T,S}(n,v)\bigr)'(t)}_{L^2_x}dt\\
&\hspace*{16em}    -\lim_{t\to \I} \LR{u(t), P_{N_1}I_{T,S}(n,v)}_{L^2_x} \Bigr|^2,\\
\lec &\sum_{N_1 \ge 1} N_1^{d-3} \sup_{\|u\|_{U^2_S}=1}\Bigl( \Bigl| \int_{\R} \int_{\R^d} \1_{[0,T]}
          nv\, \ol{u_{N_1}}dxdt \Bigr|^2 + \|u\|_{L^{\I}_tL^2_x}^2 \|P_{N_1}I_{T,S}(n,v)\|_{L^{\I}_tL^2_x}^2  \Bigr) \\
 \lec& \sum_{N_1 \ge 1} N_1^{d-3}\sup_{\|u\|_{U^2_S}=1} \Bigl|\int_{\R} \int_{\R^d} \1_{[0,T]}nv\, \ol{u_{N_1}}dxdt\Bigr|^2 
           + \sum_{N_1 \ge 1} N_1^{d-3} \|P_{N_1} I_{T,S}(n,v)\|_{L^{\I}_tL^2_x}^2  \\
 \lec& \sum_{i=1}^3J_{i,Y} + \sum_{N_1 \ge 1} N_1^{d-3} \|P_{N_1} I_{T,S}(n,v)\|_{L^{\I}_tL^2_x}^2
}
where 
\EQQS{
 J_{1,Y} &:= \sum_{N_1 \ge 1} N_1^{d-3}\sup_{\|u\|_{U^2_S}=1}\Bigl|\sum_{N_2 \sim N_1}\sum_{N_3 \ll N_2}
                 \int_{\R} \int_{\R^d} \1_{[0,T]}n_{N_3}v_{N_2}\ol{u_{N_1}}dxdt\Bigr|^2, \\
 J_{2,Y} &:= \sum_{N_1 \ge 1} N_1^{d-3}\sup_{\|u\|_{U^2_S}=1}\Bigl|\sum_{N_2 \gec N_1}\sum_{N_3 \sim N_2}
                 \int_{\R} \int_{\R^d} \1_{[0,T]}n_{N_3}v_{N_2}\ol{u_{N_1}}dxdt\Bigr|^2, \\
 J_{3,Y} &:= \sum_{N_1 \ge 1} N_1^{d-3}\sup_{\|u\|_{U^2_S}=1}\Bigl|\sum_{N_2 \ll N_1}\sum_{N_3 \sim N_1}
                 \int_{\R} \int_{\R^d} \1_{[0,T]}n_{N_3}v_{N_2}\ol{u_{N_1}}dxdt\Bigr|^2.
}
By Corollary \ref{Strichartz-S} and Remark \ref{embed}, it follows that
\EQ{
\|u\| _{E} \lec \|u\|_{U^2_S}, \ \ \ \|u\|_{V^2_S}\lec \|u\|_{U^2_S}.\label{eq3.1}
}
We obtain $J_{i,Y} \lec \|n\|_{\dot{Y}^l_{W_{\pm}}}^2\|v\|_{X^k_S}^2$ in the same manner as 
the estimates for $J_{i,E}$ with $i=1,2,3$ if we use \eqref{eq3.1}.
Collecting \eqref{J-Y0} and the estimates above, we conclude \eqref{J-Y1}.
Next, we show \eqref{J-Y2}.
By Corollary \ref{V^2_A} and Remark \ref{embed}, we have   
\EQS{ \label{ihV2}
 &\| P_{<1} I_{T,S} (n,v)\|_{V^2_S} \notag \\
 &= \sup_{\|u\|_{U^2_S}=1} \Bigl| \int_{-\I}^{\I} \LR{u(t), S(t)\bigl( S(-\cdot )P_{<1}I_{T,S}(n,v)\bigr)'(t)}_{L^2_x} dt \notag \\
&\hspace*{8em}     - \lim_{t\to \I}\LR{u(t), \bigl(P_{<1}I_{T,S}(n,v)\bigr)(t)}_{L^2_x}\Bigr|  \notag \\
 &\lec \sup_{\|u\|_{U^2_S}=1} \Bigl( \Bigl| \int_{\R} \int_{\R^d} \1_{[0,T]}nv\ol{P_{<1}u} dxdt \Bigr|  
          + \|u\|_{L^{\I}_tL^2_x}\|P_{<1}I_{T,S}(n,v)\|_{L^{\I}_tL^2_x}\Bigr)  \notag \\
 &\lec \sup_{\|u\|_{E}=1} \Bigl| \int_{\R} \int_{\R^d} \1_{[0,T]}nv\ol{P_{<1}u}dxdt \Bigr|  
          + \|P_{<1}I_{T,S}(n,v)\|_{L^{\I}_tL^2_x}. 
}
By Proposition \ref{Strichartz}, we have  
\EQS{ \label{ihV3}
 \|P_{<1}I_{T,S}(n,v)\|_{L^{\I}_tL^2_x} \lec \sup_{\|u\|_E=1}\Bigl| \int_{\R} \int_{\R^d} \1_{[0,T]}nv \ol{P_{<1}u} dxdt \Bigr|. 
}
Collecting \eqref{ihV2}, \eqref{ihV3} and \eqref{ihE2}, we obtain \eqref{J-Y2}.
From \eqref{J-Y1} and \eqref{J-Y2}, we obtain \eqref{J-Y}.
From \eqref{J-E} and \eqref{J-Y}, we conclude \eqref{B.E.S.}.

%%%%%%%%%%%%%%%%%%%%%%%%%%%%%%%
%%%%%%%% Proof for Wave %%%%%%%%%
%%%%%%%%%%%%%%%%%%%%%%%%%%%%%%%
Finally, we prove \eqref{B.E.W.}.
By Corollary \ref{U^2_A}, we only need to estimate $K_i \lec \|u\|^2_{X_S^k}\|v\|^2_{X_S^k}$ for $i=1,2,3$, where
\EQQS{
 K_1 &:=\sum_{N_3} N_3^{d-4}\sup_{\|n\|_{V^2_{W_{\pm}}}=1}\Bigl|\sum_{N_2 \gec N_3}
           \sum_{N_1 \sim N_2}\int_{\R} \int_{\R^d}\1_{[0,T]}u_{N_1}\ol{v_{N_2}  \om \, n_{N_3}}dxdt\Bigr|^2, \\
 K_2 &:=\sum_{N_3} N_3^{d-4}\sup_{\|n\|_{V^2_{W_{\pm}}}=1}\Bigl|\sum_{N_2 \sim N_3}
           \sum_{N_1 \ll N_2}\int_{\R} \int_{\R^d}\1_{[0,T]}u_{N_1}\ol{v_{N_2} \om \, n_{N_3}}dxdt\Bigr|^2, \\
 K_3 &:=\sum_{N_3} N_3^{d-4}\sup_{\|n\|_{V^2_{W_{\pm}}}=1}\Bigl|\sum_{N_2 \ll N_3}
           \sum_{N_1 \sim N_3}\int_{\R} \int_{\R^d}\1_{[0,T]}u_{N_1}\ol{v_{N_2} \om \, n_{N_3}}dxdt\Bigr|^2. 
}
%%%%%%%%%%%%%%%%%%%%%%%%%
%%%%% Estimate of K_1 %%%%%
%%%%%%%%%%%%%%%%%%%%%%%%%
By the triangle inequality, Lemma \ref{trilinear} $(i)$ and the Cauchy-Schwarz inequality, we have   
\EQQS{
 K_1^{1/2} &\lec \sum_{N_2}\sum_{N_1 \sim N_2} \Bigl{\{} \sum_{N_3 \lec N_2} N_3^{d-4}\sup_{\|n\|_{V^2_{W_{\pm}}}=1}  
                       \Bigl| \int_{\R} \int_{\R^d} \1_{[0,T]}u_{N_1}\ol{v_{N_2}} \ol{\om \, n_{N_3}} dxdt \Bigr|^2 \Bigr{\}}^{1/2} \\
              &\lec \sum_{N_2}\sum_{N_1 \sim N_2} \Bigl{\{} \sum_{N_3 \lec N_2} N_3^{d-4} 
                        ( N_3^{(d-4)/2} N_3 \|u_{N_1}\|_E \|v_{N_2}\|_E)^2 \Bigr{\}}^{1/2} \\
              &\lec \sum_{N_2}\sum_{N_1 \sim N_2} \bigl( N_2^{2d-6} \|u_{N_1}\|_E^2 \|v_{N_2}\|_E^2\bigr)^{1/2}\\
              &\lec \Big(\sum_{N}  N^{d-3} \|u_{N}\|_E^2\Big)^{1/2} 
\Big(\sum_{N} N^{d-3} \|v_{N}\|_E^2\Big)^{1/2} .
}
By Mihlin-H\"{o}rmander's multiplier theorem, it follows that
\EQ{\label{eK1}
\sum_{N<1}  N^{d-3} \|u_{N}\|_E^2
\lec \sum_{N<1}  N^{d-3} \|P_{<1}u\|_E^2
\lec \|P_{<1}u\|_E^2.
}
Thus, we conclude $K_1 \lec \|u\|^2_{X^k_S} \|v\|^2_{X^k_S}$.
%%%%%%%%%%%%%%%%%%%%%%
%%%%% Estimate of K_2 %%%%%
%%%%%%%%%%%%%%%%%%%%%%%%%
Next, we estimate $K_2$.
Put $K_2=K_{2,1}+K_{2,2}$ where
\EQQS{
&K_{2,1}:=\sum_{N_3\lec 1} N_3^{d-4}\sup_{\|n\|_{V^2_{W_{\pm}}}=1}\Bigl|\sum_{N_2 \sim N_3}
           \sum_{N_1 \ll N_2}\int_{\R} \int_{\R^d}\1_{[0,T]}u_{N_1}\ol{v_{N_2} \om \, n_{N_3}}dxdt\Bigr|^2,\\
&K_{2,2}:=\sum_{N_3\gg 1} N_3^{d-4}\sup_{\|n\|_{V^2_{W_{\pm}}}=1}\Bigl|\sum_{N_2 \sim N_3}
           \sum_{N_1 \ll N_2}\int_{\R} \int_{\R^d}\1_{[0,T]}u_{N_1}\ol{v_{N_2} \om \, n_{N_3}}dxdt\Bigr|^2.
}
By Lemma \ref{trilinear} $(i)$, we have
\EQ{\label{k221}
K_{2,1}\lec &\sum_{N_2\lec 1} N_2^{d-4}
     \Big(N_2^{(d-4)/2}N_2 \Big\| \sum_{N_1 \ll N_2}u_{N_1}\Big\|_E \|v_{N_2}\|_E\Big)^2\\
\lec &\| P_{<1} u \|^2_E \sum_{N_2\lec 1} N_2^{2d-6}\|v_{N_2}\|^2_E\\
\lec &\|u\|^2_{X^k_S} \|v\|^2_{X^k_S}.
}
For the estimate of $K_{2,2}$,
we take $M=\e N_2^2$ for sufficiently small $\e >0$. 
Then, from Lemma \ref{Recovery}, we have  
\EQQS{
 &P_{N_1}Q_{<M}^S \bigl((Q_{<M}^S{v_{N_2}})(Q_{<M}^{W_{\pm}}{\om \, n_{N_3}})\bigr) \\
 &= P_{N_1}Q_{<M}^S \Bigl[ \F^{-1}\Bigl( \int_{\ta_1 = \ta_2 + \ta_3,\, \xi_1 = \xi_2 + \xi_3} 
       \widehat{(Q_{<M}^S {v_{N_2}})}(\ta_2, \xi_2)
       \widehat{(Q_{<M}^{W_{\pm}}{\om \, n_{N_3}})}(\ta_3, \xi_3)  \Bigr) \Bigr]
 =0,
}
when $N_2 \gg \LR{N_1}$.
Therefore, 
\EQQS{
P_{N_1}(v_{N_2} \om \, n_{N_3}) = \sum_{i=1}^4 P_{N_1}G_i, 
}
where 
\EQQS{
 &G_1 :=  {v_{N_2}(Q_{\ge M}^{W_{\pm}}\om \, n_{N_3})}, \qquad \qquad \ \ \, 
 G_2 := {Q_{\ge M}^S\bigl( v_{N_2}(Q_{<M}^{W_{\pm}}\om \, n_{N_3})} \bigr), \notag \\
 &G_3 := {(Q_{\ge M}^S v_{N_2})(Q_{<M}^{W_{\pm}}\om \, n_{N_3})} ,\qquad
 G_4 := {-Q_{\ge M}^S\bigl((Q_{\ge M}^Sv_{N_2})(Q_{<M}^{W_{\pm}}\om \, n_{N_3}) \bigr)}.
}
Therefore, it follows that $$K_{2,2}\le K_{2,2}^{(1)}+K_{2,2}^{(2)}+K_{2,2}^{(3)}+K_{2,2}^{(4)}$$
where
\EQQ{
 K_{2,2}^{(j)} &:= \sum_{N_3\gg 1}N_3^{d-4}\sup_{\|n\|_{V^2_{W_{\pm}}}=1} \Bigl|\sum_{N_2 \sim N_3}
                        \sum_{N_1 \ll N_2}\int_{\R} \int_{\R^d} \1_{[0,T]}u_{N_1} \ol{G_j} dxdt\Bigr|^2
}
for $j=1,2,3,4$.
Note that $N_3 \gg 1$ and $N_2\sim N_3$ implies $N_2 >2^2$.
By Lemma \ref{trilinear} $(iii)$ and \eqref{ct}, we have   
\EQS{ \label{k2a}
 K_{2,2}^{(1)} &\lec \sum_{N_2 > 2^2}N_2^{d-4}\bigl( N_2^{-1/2}N_2 \|u\|_{Y^k_S} \|v_{N_2}\|_E \bigr)^2 \notag \\
                  &\lec \sum_{N_2 > 2^2} N_2^{d-3} \|u\|_{Y^k_S}^2 \|v_{N_2}\|_E^2 
                   \lec \|u\|_{Y^k_S}^2 \|v\|_{E^k}^2.
}
We apply Lemma \ref{Mhigh}, Lemma \ref{trilinear} $(v)$ and \eqref{ct}, then we have 
\EQS{ \label{k2b}
 K_{2,2}^{(2)} \lec \sum_{N_2 > 2^2}N_2^{d-4}\bigl( N_2^{-1/2}N_2\|u\|_{Y^k_S}\|v_{N_2}\|_E \bigr)^2
                  \lec \|u\|_{Y^k_S}^2\|v\|_{E^k}^2.
}
By Lemma \ref{trilinear} $(iv)$, we have  
\EQS{ \label{k2c}
 K_{2,2}^{(3)} \lec \sum_{N_2 > 2^2}N_2^{d-4}\bigl( N_2^{-1/2} N_2 \|u\|_{E^k} \|v_{N_2}\|_{V^2_S} \bigr)^2 
                  \lec \|u\|_{E^k}^2\|v\|_{Y^k_S}^2.
}
Applying Lemma \ref{Mhigh}, Lemma \ref{trilinear} $(vi)$ and \eqref{ct}, we obtain 
\EQ{ \label{k2d}
 K_{2,2}^{(4)} \lec \sum_{N_2 > 2^2}N_2^{d-4}\bigl( N_2^{-1/2}N_2\|u\|_{Y^k_S}\|v_{N_2}\|_{V^2_S} \bigr)^2 
                  \lec \|u\|_{Y^k_S}^2\|v\|_{Y^k_S}^2.
}
Hence, collecting \eqref{k221}, \eqref{k2a}, \eqref{k2b}, \eqref{k2c} and \eqref{k2d}, 
we have $K_2 \lec \|u\|_{X^k_S}^2\|v\|_{X^k_S}^2$.
By symmetry, we also obtain $K_3\lec \|u\|_{X^k_S}^2\|v\|_{X^k_S}^2$ in the same manner as the estimate of $K_2$.
\end{proof}
%%%%%%%%%%%%%%%%%%%%%%%%%%%%%%%%%%%%%%%%%%%%%%%%%%%%%%%%%%%%%%%%%%%%%%%%%%%%%%%%%%%%%%%%%%%%%%%%%%%%%%%%%%
%%%%%%%%%%%%%%%%%%%%%%%%%%%%%%%%%%%%%%%%%%%%%%%%%%%%%%%%%%%%%%%%%%%%%%%%%%%%%%%%%%%%%%%%%%%%%%%%%%%%%%%%%%
Next, we consider the inhomogeneous case.
\begin{cor}  \label{B.E.S.W.}
Let $d \ge 4,\ k=(d-3)/2$ and $l=(d-4)/2.$ Then for all $0<T<\I$, it holds that
\EQS{
&\| I_{T,S}(n,v)\|_{X^k_S}\lec \| n\|_{{Y}^l_{W_{\pm}}}\| v\|_{X^k_S}, \label{H.B.E.S.} \\
&\|I_{T,\, W_{\pm}}(u,v)\|_{Z^l_{W_{\pm}}} \lec \|u\|_{X^k_S}\|v\|_{X^k_S}.\label{H.B.E.W.}
} 
\end{cor}
\begin{proof}
First we consider \eqref{H.B.E.W.}.
From Proposition \ref{B.E.homo}, we have
\EQQ{
\Bigl{\|}\sum_{N \ge 1} P_N I_{T,\, W_{\pm}}(u,v)\Bigr{\|}_{Z^l_{W_{\pm}}} 
   \sim \Bigl{\|}\sum_{N \ge 1} P_N I_{T,\, W_{\pm}}(u,v)\Bigr{\|}_{\dot{Z}^l_{W_{\pm}}}
   \lec \|u\|_{X^k_S}\|v\|_{X^k_S}.
}
Hence, we only need to show the following. 
\EQS{
 \|P_{<1}I_{T,W_{\pm}}(u,v)\|_{U^2_{W_{\pm}}} \lec \|u\|_{X^k_S}\|v\|_{X^k_S}. \label{ihU}
}
By Corollary \ref{U^2_A} and H\"older's inequality, we have 
\EQS{ 
 (\text{LHS of}\ \eqref{ihU}) 
 &= \sup_{\|n\|_{V^2_{W_{\pm}}}=1} \Bigl| \int_{\R} \int_{\R^d} \1_{[0,T]}u\, \bar{v}\, \ol{P_{<1}\om \, n}\, dxdt \Bigr|  \notag \\
 &\lec \sup_{\|n\|_{V^2_{W_{\pm}}}=1} \|u\|_E\|v\|_E\|P_{<1}\om \, n\|_{L^{\I}_tL^{d/2}_x}.
          \label{ihU2}
}
Since
\EQQ{
\|u\|_E\le \|P_{<1} u\|_{E}+\Big(\sum_{N\ge 1}N^{-2k}\Big)^{1/2}\Big(\sum_{N \ge 1}N^{2k}\|P_N u\|_E^2\Big)^{1/2}
}
by the Cauchy-Schwarz inequality, we have 
\EQS{
 \|u\|_E \lec \|u\|_{X^k_S},\ \ \ \ \  \|v\|_E \lec \|v\|_{X^k_S}.  \label{ihU3}
}
By the Sobolev inequality and Remark \ref{embed}, we have
\EQ{ \label{ihU4}
 \|P_{<1}\om \, n\|_{L^{\I}_tL^{d/2}_x} 
  \lec \bigl{\|} |\na_x|^{(d-4)/2} P_{<1}\om \, n \bigr{\|}_{L^{\I}_tL^2_x}
  \lec \|P_{<1}n\|_{L^{\I}_tL^2_x}
  \lec \|n\|_{V^2_{W_{\pm}}}. 
}
Hence, collecting \eqref{ihU2}, \eqref{ihU3} and \eqref{ihU4}, we obtain \eqref{ihU}.

Next, we consider \eqref{H.B.E.S.}.
From \eqref{B.E.S.}, we obtain
\EQQ{
\Big\| I_{T,S}\Big(\Big(\sum_{1 \lec N} P_N n\Big),v\Big)\Big\|_{X^k_S}
  \lec \Big\| \sum_{1 \lec N} P_N n \Big\|_{\dot{Y}^l_{W_{\pm}}}\| v\|_{X^k_S}
  \lec \| n\|_{{Y}^l_{W_{\pm}}}\| v\|_{X^k_S}.
}
Therefore, we only need to show
\EQ{\label{B.E.S.L.}
\Big\| I_{T,S}\Big(\sum_{N\ll 1} P_N n,v\Big) \Big\|_{X^k_S}\lec \| n\|_{{Y}^l_{W_{\pm}}}\| v\|_{X^k_S}.
}
Note that \eqref{B.E.S.L.} easily follows from \eqref{B.E.S.} for $d\ge 5$ because
$l>0$ and
\EQQ{
\Big\| \sum_{N\ll 1} P_N n \Big\|_{\dot{Y}^l_{W_{\pm}}} 
 &\lec \Big( \sum_{N\ll 1} N^{2l} \|P_N n \|^2_{V^2_{W_{\pm}}}\Big)^{1/2}\\
&\lec \|P_{<1} n\|_{V^2_{W_{\pm}}} \Big( \sum_{N\ll 1} N^{2l}\Big)^{1/2}\\
&\lec \| n\|_{{Y}^l_{W_{\pm}}}.
}
However, we need more computation for $d=4$.
We show \eqref{B.E.S.L.} by an almost same manner as the proof of \eqref{B.E.S.}. By the definition of $X^k_S$ norm, we only need to show
\EQS{
&\Big\| P_{<1} I_{T,S}\Big(\sum_{N\ll 1} P_N n,v\Big) \Big\|_{E}\lec \| n\|_{{Y}^l_{W_{\pm}}}\| v\|_{X^k_S}\label{B.E.S.L.1}\\
&\Big(\sum_{N_1\ge 1}N_1^{2k}\Big\| P_{N_1} I_{T,S}\Big(\sum_{N\ll 1} P_N n,v\Big) \Big\|^2_{E}\Big)^{1/2}\lec \| n\|_{{Y}^l_{W_{\pm}}}\| v\|_{X^k_S}\label{B.E.S.L.2}\\
&\Big\| P_{<1} I_{T,S}\Big(\sum_{N\ll 1} P_N n,v\Big) \Big\|_{V^2_S}\lec \| n\|_{{Y}^l_{W_{\pm}}}\| v\|_{X^k_S}\label{B.E.S.L.3}\\
&\Big(\sum_{N_1\ge 1}N_1^{2k}\Big\| P_{N_1} I_{T,S}\Big(\sum_{N\ll 1} P_N n,v\Big) \Big\|^2_{V^2_S}\Big)^{1/2}\lec \| n\|_{{Y}^l_{W_{\pm}}}\| v\|_{X^k_S}.\label{B.E.S.L.4}
}
Since
\EQQ{
\Big\|\sum_{N\ll 1} P_N n\Big\|_{L_t^\infty L_x^{d/2}}\lec
\Big\| |\nabla_x|^{(d-4)/2}\sum_{N\ll 1} P_N n\Big\|_{L_t^\infty L_x^{2}}\lec
\| P_{<1} n\|_{V^2_{W_{\pm}}}\lec \|n\|_{Y^l_{W_{\pm}}},
}
we have \eqref{B.E.S.L.1} in the same manner as \eqref{ihE2}.
We also have \eqref{B.E.S.L.3} in the same manner as \eqref{ihV2}, \eqref{ihV3} and  \eqref{ihE2}.
Since $P_{N_3}\sum_{N\ll 1} P_N n=0$ for $N_3 \gec 1$, the left-hand side of \eqref{B.E.S.L.2} is equal to $J^*_{1,E}$ where
\EQQ{
J^*_{1,E}:=& \biggl( \sum_{N_1 \ge 1} N_1^{d-3} \Bigl{\|}
 \int_0^t \1_{[0,T]}(t')S(t-t')\\
&\cross \sum_{N_2 \sim N_1} \sum_{N_3 \ll N_2} P_{N_1} \Bigl(P_{N_3}\Big(\sum_{N\ll 1} P_N n\Big)(t') v_{N_2}(t')\Bigr) dt' \Bigr{\|}_E^2\biggr)^{1/2}.
}
We obtain
\EQ{\label{Jstar}
J^*_{1,E} \lec \Big\| \sum_{N\ll 1} P_N n \Big\|_{{Y}^l_{W_{\pm}}}\| v\|_{X^k_S} \lec \| n\|_{{Y}^l_{W_{\pm}}}\| v\|_{X^k_S}
}
in the same manner as the estimate of $J_{1,E}$ by using $(iib)$ of Lemma \ref{trilinear} instead of $(iia)$.
Thus, we obtain \eqref{B.E.S.L.2}. 
We also obtain \eqref{B.E.S.L.4} in the same manner as \eqref{J-Y1} by using \eqref{Jstar} instead of the estimate for $J_{1,E}$.
\end{proof}

%%%%%%%%%%%%%%%%%%%%%%%%%%%%%%%%%%%%%%%%%%%%%%%%%%%%%%%%%%%%%%%%%%%%%%%%%%%%%%%%%%%%%%%%%%%%%%
%%%%%%%%%%%%%%%%%%%%%%%%%%%%%%%%%%%%%%%%%%%%%%%%%%%%%%%%%%%%%%%%%%%%%%%%%%%%%%%%%%%%%%%%%%%%%%
%%%%%%%%%%%%%%%%%%%%%%%%%%%%%%%%%%%%%%%%  Section 4  %%%%%%%%%%%%%%%%%%%%%%%%%%%%%%%%%%%%%%%%%%%
%%%%%%%%%%%%%%%%%%%%%%%%%%%%%%%%%%%%%%%%%%%%%%%%%%%%%%%%%%%%%%%%%%%%%%%%%%%%%%%%%%%%%%%%%%%%%%
%%%%%%%%%%%%%%%%%%%%%%%%%%%%%%%%%%%%%%%%%%%%%%%%%%%%%%%%%%%%%%%%%%%%%%%%%%%%%%%%%%%%%%%%%%%%%%
\section{the proof of the main theorem}
By the Duhamel principle, 
we consider the following integral equation corresponding to \eqref{redn} on the time interval 
$[0, T]$ with $0< T< \I:$
\EQ{\label{integral_eq}
  (u,n_\pm)= (\Phi_1(u, n_{\pm}), \Phi_{2\pm}(u)), 
}
where
\EQQS{
  &\Phi_1(u, n_{\pm}) := S(t)u_0+I_{T,S}(n_+, u)(t)+I_{T,S}(n_-, u)(t), \\
  &\Phi_{2\pm}(u) := W_{\pm}(t)n_{\pm 0}+I_{T,W_{\pm}}(u, u)(t).
}
\begin{prop}\label{main_prop1}
    Let $d \ge 4, k=(d-3)/2$ and $l=(d-4)/2$. \\
 (i) (existence) 
Let $\de > 0$ be sufficiently small.
Then, for any $0<T<\infty$ and any initial data 
$$(u_0, n_{\pm 0}) \in B_{\de }(H^k(\R^d) \cross \dot{H}^l(\R^d)) \
      (\text{resp. } B_{\de }(H^k(\R^d) \cross H^l(\R^d))),$$
there exists a solution to \eqref{integral_eq} on $[0,T]$ satisfying
\EQQ{
&(u, n_{\pm}) \in X^k_S([0, T ]) \cross \dot{Y}^l_{W_{\pm}}([0, T ])
      \subset C([0, T ]; H^k(\R^d)) \cross C([0, T ]; \dot{H}^l(\R^d))\\
&\text{(resp.} \ (u, n_{\pm}) \in  X^k_S([0, T ]) \cross Y^l_{W_{\pm}}([0, T ]) \subset C([0, T ]; H^k(\R^d)) \cross C([0, T ]; H^l(\R^d))\text{).}
}
 (ii) (uniqueness) 
Let
$$(u, n_{\pm}), (v, m_{\pm}) \in X^k_S([0, T ]) \cross \dot{Y}^l_{W_{\pm}}([0, T ]) \ (\text{resp. } \in X^k_S([0, T]) \cross Y^l_{W_{\pm}}([0, T]))$$
be solutions to \eqref{integral_eq} on $[0,T]$ for some $T>0$
with the same initial data.
Then $(u(t), n_{\pm}(t))=(v(t), m_{\pm}(t))$ on $t\in [0,T]$. \\
 (iii) (continuous dependence of the solution on the initial data) 
    The flow map obtained by (i):
$$B_{\de }(H^k(\R^d) \cross \dot{H}^l(\R^d)) 
 \ni (u_0, n_{\pm 0}) \mapsto
           (u, n_{\pm})\in X^k_S([0, T ]) \cross \dot{Y}^l_{W_{\pm}}([0, T ])$$
$$(\text{resp. } B_{\de }(H^k(\R^d) \cross H^l(\R^d)) 
           \ni (u_0, n_{\pm 0}) \mapsto
           (u, n_{\pm})\in X^k_S([0, T ]) \cross Y^l_{W_{\pm}}([0, T ]) )$$
    is Lipschitz continuous.  \\
 (iv) (persistence)
For any $a \ge 0$, there exists $\de=\de(a)>0$ such that if
$$(u_0, n_{\pm 0}) \in B_{\de }(H^k(\R^d) \cross H^l(\R^d)) \cap H^{k+a}(\R^d) \cross H^{l+a}(\R^d),$$
then the solution $(u, n_{\pm})$ obtained by (i) is in
$$X^{k+a}_S([0, T ]) \cross Y^{l+a}_{W_{\pm}}([0, T]) \subset C([0,T];H^{k+a}(\R^d))\cross C([0,T];H^{l+a}(\R^d))$$ for any $0<T<\infty$.
\end{prop}
\begin{rem}\label{time_rev}
Due to the time reversibility of the Zakharov equation, Proposition \ref{main_prop1} holds on corresponding time interval $[-T,0]$.
\end{rem}
\begin{rem}\label{global_sol}
By $(i)$ in Proposition \ref{main_prop1} and Remark \ref{time_rev},
we have solutions to \eqref{integral_eq} on $[0,T]$ and $[-T,0]$ for any $T>0$. Since we can take any large $T$ and have the uniqueness, 
the solution $(u(t), n_{\pm}(t)) \in C((-\infty,\infty);H^k(\R^d))\times C((-\infty,\infty);\dot{H}^l(\R^d))$ (resp. $C((-\infty,\infty);H^k(\R^d))\times C((-\infty,\infty);{H}^l(\R^d))$) can be defined uniquely when $(u_0, n_{\pm 0}) \in B_{\de }(H^k(\R^d) \cross \dot{H}^l(\R^d)) \ (\text{resp. } B_{\de }(H^k(\R^d) \cross H^l(\R^d)))$.
\end{rem}
\begin{prop}(scattering)\label{main_prop2}
Let $(u(t), n_{\pm}(t))$ be the solution to \eqref{integral_eq} with $(u_0, n_{\pm 0}) \in B_{\de }(H^k(\R^d) \cross \dot{H}^l(\R^d))$ on $(-\infty,\infty)$
obtained by Proposition \ref{main_prop1}, Remark \ref{time_rev}
and Remark \ref{global_sol}.
Then, there exist $(u_{+\I}, n_{\pm, +\I})$ and $(u_{-\I}, n_{\pm, -\I})$
in $H^k(\R^d)\times \dot{H}^l(\R^d)$ such that
\EQQ{
\|u(t)-S(t)u_{+\I}\|_{H^k}+\|n_{\pm}(t)-W_\pm(t)n_{\pm, +\I}\|_{\dot{H}^l}\to 0
}
as $t \to \I$ and
\EQQ{
\|u(t)-S(t)u_{-\I}\|_{H^k}+\|n_{\pm}(t)-W_\pm(t)n_{\pm, -\I}\|_{\dot{H}^l}\to 0
}
as $t \to -\I$.
The similar result holds for the inhomogeneous case.
\end{prop}

\begin{proof}[Proof of Proposition \ref{main_prop1}]
We will show only the case $(u_0, n_{\pm 0}) \in B_{\de }(H^k(\R^d) \cross H^l(\R^d))$ because the proof of the case $(u_0, n_{\pm 0})  \in B_{\de }(H^k(\R^d) \cross \dot{H}^l(\R^d))$ follows from the same argument if we use \eqref{B.E.W.} instead of Corollary \ref{B.E.S.W.}.

First, we prove $(i)$.
%%%%%%%%%%%%%%%%%%%%%%%%%%%%%%%%%%%%%%%%%%%%
%%%%% \Phi_{i} (i=1,2) are contraction mapping %%%%%
%%%%%%%%%%%%%%%%%%%%%%%%%%%%%%%%%%%%%%%%%%%%
We denote $I := [0, T]$ and
\EQQ{
\chi(t):=
\begin{cases}
0 \ \ \text{for} \ \ t<-1,\\
t+1 \ \text{for} \ \ -1 \le t \le 0,\\
1 \ \ \text{for} \ \ 0<t.
\end{cases}
}
By Proposition \ref{Strichartz} and the definition of $X_S^k, Y_{W_\pm}^l$,
it follows that
$\chi(t) S(t)u_0\in X^k_S$, $\chi(t)W_{\pm}(t)n_{\pm 0} \in Y^l_{W_{\pm}}$
and there exists $C>0$ such that     
\EQQ{
 \|\chi(t)S(t)u_0\|_{X^k_S} \le C\|u_0\|_{H^k}, \
 \|\chi(t)W_{\pm}(t)n_{\pm 0}\|_{{Y}^l_{W_{\pm}}} \le C\|n_{\pm 0}\|_{{H}^l}.
}
Since $\chi(t)=1$ on $I$, we obtain
\EQQ{
&S(t)u_0\in X^k_S(I), \ \ W_{\pm}(t)n_{\pm 0} \in Y^l_{W_{\pm}}(I),\\
&\|S(t)u_0\|_{X^k_S(I)} \le C\|u_0\|_{H^k}, \ \
 \|W_{\pm}(t)n_{\pm 0}\|_{{Y}^l_{W_{\pm}}(I)} \le C\|n_{\pm 0}\|_{{H}^l}.
}
Assume that $(u_0, n_{\pm 0}) \in B_{\de}(H^k(\R^d)\times {H}^l(\R^d)),\ 
(u, n_{\pm}) \in B_r(X^k_S(I)\cross {Y}^l_{W_{\pm}}(I))$. 
Then, by Proposition \ref{B.E.homo}, Corollary \ref{B.E.S.W.} and $\|\cdot\|_{Y^l_{W_{\pm}}} \lec \|\cdot\|_{Z^l_{W_{\pm}}}$, we have  
\EQQS{
&\|\Phi_1(u, n_{\pm})\|_{X^k_S}=\|\Phi_1(u^*, n^*_{\pm})\|_{X^k_S}
  \le C\|u_0\|_{H^k}+C\|n^*_{\pm}\|_{Y^l_{W_{\pm}}}\|u^*\|_{X^k_S}\\
&\|\Phi_{2\pm}(u)\|_{Y^l_{W_{\pm}}}=\|\Phi_{2\pm}(u^*)\|_{Y^l_{W_{\pm}}}
  \le C\|n_{\pm 0}\|_{H^l}+C\|u^*\|^2_{X^k_S}
}
for any $u^*\in X^k_S, n^*_{\pm}\in Y^l_{W_{\pm}}$ satisfying 
$u^*(t)=u(t), n^*_{\pm}(t)=n_{\pm}(t)$ on $I$.
Therefore,
\EQQS{
&\Phi_1(u, n_{\pm}) \in X^k_S(I), \ \ \ \Phi_{2\pm}(u) \in Y^l_{W_{\pm}}(I),\\ 
&\|\Phi_1(u, n_{\pm})\|_{X^k_S(I)}
  \le C\|u_0\|_{H^k}+C\|n_{\pm}\|_{Y^l_{W_{\pm}}(I)}\|u\|_{X^k_S(I)} 
  \le C\de + Cr^2,\\
&\|\Phi_{2\pm}(u)\|_{Y^l_{W_{\pm}}(I)}
  \le C\|n_{\pm 0}\|_{H^l}+C\|u\|^2_{X^k_S(I)} 
  \le C\de +Cr^2.
}
We choose $\de =r^2,\ r=1/4C,$ then we have   
\EQQS{
\|\Phi_1(u, n_{\pm})\|_{X^k_S(I)}+ \|\Phi_{2\pm}(u)\|_{Y^l_{W_{\pm}}(I)} \le r.
}
Hence, $(\Phi_1,\Phi_{2\pm})$ is a map from $B_r( X^k_S([0, T ]) \cross {Y}^l_{W_{\pm}}([0, T ]))$ into itself. 
Note that $r$ does not depend on $T$.
Moreover, we assume $(v, m_{\pm})\in B_r(X^k_S(I) \cross {Y}^l_{W_{\pm}}(I)),$ then 
\EQS{
  &\|\Phi_1(u, n_{\pm})-\Phi_1(v, m_{\pm})\|_{X^k_S(I)} \notag \\
  &=\|I_{T,S}(n_{\pm}, u)(t)-I_{T,S}(m_{\pm}, v)(t)\|_{X^k_S(I)} \notag \\
  &\le \|I_{T,S}(n_{\pm}, u-v)\|_{X^k_S(I)}+\|I_{T,S}(n_{\pm}-m_{\pm}, v)\|_{X^k_S(I)} \notag \\
  &\le C(\|n_{\pm}\|_{{Y}^l_{W_{\pm}}(I)}\|u-v\|_{X^k_S(I)}+\|n_{\pm}-m_{\pm}\|_{{Y}^l_{W_{\pm}}(I)}\|v\|_{X^k_S(I)}) \label{con1}\\
  &\le (1/4)(\|u-v\|_{X^k_S(I)}+\|n_{\pm}-m_{\pm}\|_{{Y}^l_{W_{\pm}}(I)}), \notag \\
 &\|\Phi_{2\pm}(u)-\Phi_{2\pm}(v)\|_{{Y}^l_{W_{\pm}}(I)} \notag \\
  &=\|I_{T,W_{\pm}}(u, u)(t)-I_{T,W_{\pm}}(v, v)(t)\|_{{Y}^l_{W_{\pm}}(I)} \notag \\
  &\le C(\|u\|_{X^k_S(I)}+\|v\|_{X^k_S(I)})\|u-v\|_{X^k_S(I)} \label{con2} \\
  &\le (1/2)\|u-v\|_{X^k_S(I)}. \notag 
}
Therefore, $(\Phi_1, \Phi_{2\pm})$ is a contraction mapping on $B_r( X^k_S([0, T ]) \cross {Y}^l_{W_{\pm}}([0, T ]))$. 
Thus, by the Banach fixed point theorem, we have a solution to \eqref{integral_eq} in it.

Next, we prove $(ii)$ by contradiction. 
Let $(u, n_{\pm}),\ (v, m_{\pm}) \in X^k_S([0,T]) \cross {Y}^l_{W_{\pm}}([0,T])$ are two solutions satisfying 
$(u(0), n_{\pm}(0))=(v(0), m_{\pm}(0))$. 
Assume that
\EQQS{
 T' := \sup \{0 \le t < T \, ; u(t)=v(t), n_{\pm}(t)=m_{\pm}(t) \} <T. 
}
By a translation in $t$, it suffices to consider the case $T'=0$. 
Let $0 < \ta \le T$. From \eqref{con1} and Proposition \ref{unique}, we obtain 
\EQQS{
 \|u-v\|_{X^k_S([0,\ta ])} 
 &\le C(\|n_{\pm}\|_{Y^l_{W_{\pm}}([0, \ta ])}\|u-v\|_{X^k_S([0, \ta ])}
          +\|n_{\pm}-m_{\pm}\|_{Y^l_{W_{\pm}}([0, \ta ])}\|v\|_{X^k_S([0, \ta ])}) \\
 &\le (1/4)\big(\|u-v\|_{X^k_S([0, \ta ])} + \|n_{\pm}-m_{\pm}\|_{Y^l_{W_{\pm}}([0, \ta ])}\big).
}
Here, we took sufficiently small $\tau$.
Hence, we obtain 
\EQS{ \label{uniq1}
 \|u-v\|_{X^k_S([0,\ta ])} \le (1/3)\|n_{\pm}-m_{\pm}\|_{Y^l_{W_{\pm}}([0, \ta ])}. 
}
Similarly, by \eqref{con2} and Proposition \ref{unique}, we obtain 
\EQS{
 \|n_{\pm}-m_{\pm}\|_{Y^l_{W_{\pm}}([0, \ta ])} 
 &\le C(\|u\|_{X^k_S([0, \ta ])}+\|v\|_{X^k_S([0, \ta ])})\|u-v\|_{X^k_S([0, \ta ])} \notag \\
 &\le (1/2)\|u-v\|_{X^k_S([0, \ta ])}.   \label{uniq2}
}
Here, we took sufficiently small $\tau$.
Hence, from \eqref{uniq1} and \eqref{uniq2}, we obtain
$u(t)=v(t), n_{\pm}(t)=m_{\pm}(t)$ on $[0,\tau]$, which contradicts to the definition of $T'$.

We omit the proof of $(iii)$ because it follows from the standard argument.
Finally, we prove $(iv)$. Fix $0<T<\infty$. Since $\LR{\x}^a \le C(a)(\LR{\x-\x_1}^a+\LR{\x_1}^a)$, we easily have   
\EQS{
 &\|I_{T,\, S}(n_\pm,u)\|_{X^{k+a}_S} \le C(a) \big( \|n_\pm\|_{Y^{l+a}_{W_{\pm}}}\|u\|_{X^k_S} + \|n_\pm\|_{Y^l_{W_{\pm}}} \|u\|_{X^{k+a}_S}\big), \label{per1}\\
 &\|I_{T,W_{\pm}}(u,u)\|_{Z^{l+a}_{W_{\pm}}} \le C(a) \|u\|_{X^{k+a}_S} \|u\|_{X^k_S}, \label{per2}
}
from Proposition \ref{B.E.homo} and Corollary \ref{B.E.S.W.}. Thus, by a similar argument as $(i)$, we obtain
\EQQ{
&\|u\|_{X^{k+a}_S(I)}\le C\|u_0\|_{H^{k+a}}+C(a) r(\|u\|_{X^{k+a}_S(I)}+\|n_+\|_{Y^{l+a}_{W_+}(I)}+\|n_-\|_{Y^{l+a}_{W_-}(I)}),\\
&\|n_{\pm}\|_{Y^{l+a}_{W_\pm}(I)}\le C\|n_{\pm 0}\|_{H^{l+a}}+C(a) r\|u\|_{X^{k+a}_S}
}
for the solution to \eqref{integral_eq} such that $(u,n_{\pm})\in B_r(X^{k}_S(I)\cross Y^{l}_{W_\pm}(I))$ with $r:=1/4C(a)$ and $\de:=r^2$.
Thus, we conclude
$$\|u\|_{X^{k+a}_S(I)}+\|n_{\pm}\|_{Y^{l+a}_{W_\pm}(I)} \le C(\|u_0\|_{H^{k+a}}+\|n_{\pm 0}\|_{H^{l+a}}).$$
\end{proof}

Finally, we prove Proposition 4.2. 
%%%%%%%%%%%%%%%%%%%%%%%%%%%%%%%%%%%%%%%%%%%%%%%%%%%%%%%%%%%%%%%%%%%%%%%%%%%%%%%%%%%%%%%%%%%%%%%%%%%
%%%%%%%%%%%%%%%%%%%%%%%%%%%%%%%%%%%%%%%%  Scattering  %%%%%%%%%%%%%%%%%%%%%%%%%%%%%%%%%%%%%%%%%%%%%%%
%%%%%%%%%%%%%%%%%%%%%%%%%%%%%%%%%%%%%%%%%%%%%%%%%%%%%%%%%%%%%%%%%%%%%%%%%%%%%%%%%%%%%%%%%%%%%%%%%%%

\begin{proof}
Since $r$ in the proof of Proposition \ref{main_prop1} does not depend on $T$,
it follows that
\EQQ{
&\|u\|_{X_S^k([0,T])}+\|n_\pm\|_{Y^l_{W_\pm}([0,T])} <M,\\
&\|u\|_{X_S^k([-T,0])}+\|n_\pm\|_{Y^l_{W_\pm}([-T,0])} <M\\
}
for any $T>0$, where the constant $M$ does not depend  on $T$.
For any $\{t_j\}_{j=0}^K\in \mathcal{Z}_0$ with $t_K<\infty$, we can take $0<T<\infty$ such that $-T<t_0$ and $t_K<T$. Then, by Lemma \ref{estX}, we have
\EQQ{
&\Big(\sum_{j=1}^K \|\LR{\nabla_x}^k(S(-t_j)u(t_j)-S(-t_{j-1})u(t_{j-1}))\|_{L^2}^2\Big)^{1/2}\\
\lec& \|\LR{\nabla_x}^k u \|_{V^2_S([0,T])}+\|\LR{\nabla_x}^k u \|_{V^2_S([-T,0])}\\
\lec &\|u\|_{X_S^k([0,T])}+\|u\|_{X_S^k([-T,0])}<2M.
}
Therefore, we have
\EQQ{
\sup_{\{t_j\}_{j=0}^K\in \mathcal{Z}_0}\Big(\sum_{j=1}^K \|\LR{\nabla_x}^kS(-t_j)u(t_j)-\LR{\nabla_x}^kS(-t_{j-1})u(t_{j-1})\|_{L^2}^2\Big)^{1/2} \lec M.
}
By Proposition \ref{embedding}, $f_{\pm}:=\lim_{t\to \pm \infty} \LR{\nabla_x}^kS(-t)u(t)$ exists in $L^2$. Put $u_{\pm \infty}:=\LR{\nabla_x}^{-k}f_{\pm}$. Then, we conclude
$$\|\LR{\nabla_x}^kS(-t)u(t) -f_{\pm} \|_{L^2}=\|u(t)-S(t)u_{\pm \infty}\|_{H^k}\to 0 \ \ \ \ \ \text{as} \ \ t\to \pm\infty.$$
Similarly, we obtain the scattering result for the wave equation. 
\end{proof}

\end{document}